\documentclass[12pt]{amsart}       
\usepackage{txfonts}
\usepackage{amssymb}
\usepackage{eucal}
\usepackage{graphicx}
\usepackage{amsmath}
\usepackage{amscd}
\usepackage[all]{xy}           
\usepackage{amsfonts,latexsym}
\usepackage{xspace}
\usepackage{epsfig}
\usepackage{float}
\usepackage{color}
\usepackage{fancybox}
\usepackage{colordvi}
\usepackage{multicol}
\usepackage{colordvi}
\usepackage[colorlinks,final,backref=page,hyperindex,hypertex]{hyperref}
\usepackage[active]{srcltx} 
\usepackage{tikz}

\usepackage{graphicx}
\usepackage{bbm}
\newtheorem{theorem}{Theorem}[section]

\theoremstyle{definition}
\newtheorem{defn}[theorem]{Definition}
\newtheorem{lemma}[theorem]{Lemma}
\newtheorem{coro}[theorem]{Corollary}
\newtheorem{prop-def}{Proposition-Definition}[section]
\newtheorem{coro-def}{Corollary-Definition}[section]

\newtheorem{remark}[theorem]{Remark}

\newtheorem{exam}[theorem]{Example}

\newcommand{\nc}{\newcommand}
\nc{\tred}[1]{\textcolor{red}{#1}}
\nc{\tblue}[1]{\textcolor{blue}{#1}}
\nc{\tgreen}[1]{\textcolor{green}{#1}}
\nc{\tpurple}[1]{\textcolor{purple}{#1}}
\nc{\btred}[1]{\textcolor{red}{\bf #1}}
\nc{\btblue}[1]{\textcolor{blue}{\bf #1}}
\nc{\btgreen}[1]{\textcolor{green}{\bf #1}}
\nc{\btpurple}[1]{\textcolor{purple}{\bf #1}}
\nc{\NN}{{\mathbb N}}
\nc{\ncsha}{{\mbox{\cyr X}^{\mathrm NC}}} \nc{\ncshao}{{\mbox{\cyr
X}^{\mathrm NC}_0}}

\newcommand{\efootnote}[1]{}

\renewcommand{\textbf}[1]{}

\newcommand{\delete}[1]{}

\nc{\mlabel}[1]{\label{#1}}  
\nc{\mcite}[1]{\cite{#1}}  
\nc{\mref}[1]{\ref{#1}}  
\nc{\mbibitem}[1]{\bibitem{#1}} 

\delete{
\nc{\mlabel}[1]{\label{#1}{\hfill \hspace{1cm}{\bf{{\ }\hfill(#1)}}}}
\nc{\mcite}[1]{\cite{#1}{{\bf{{\ }(#1)}}}}  
\nc{\mref}[1]{\ref{#1}{{\bf{{\ }(#1)}}}}  
\nc{\mbibitem}[1]{\bibitem[\bf #1]{#1}} 
}

\nc{\opa}{\ast} \nc{\opb}{\odot} \nc{\op}{\bullet} \nc{\pa}{\frakL}
\nc{\arr}{\rightarrow} \nc{\lu}[1]{(#1)} \nc{\mult}{\mrm{mult}}
\nc{\diff}{\mathfrak{Diff}}
\nc{\opc}{\sharp}\nc{\opd}{\natural}
\nc{\ope}{\circ}
\nc{\dpt}{\mathrm{d}}
\nc{\hck}{H_{RT}}
\nc{\vdf}{\calf}
\nc{\ldf}{\calf_\ell}
\nc{\hlf}{H_\ell}
\nc{\onek}{\mathbf{1}_\bfk}
\nc{\diam}{alternating\xspace}
\nc{\Diam}{Alternating\xspace}
\nc{\cdiam}{canonical alternating\xspace}
\nc{\Cdiam}{Canonical alternating\xspace}
\nc{\AW}{\mathcal{A}}

\nc{\ari}{\mathrm{ar}}

\nc{\lef}{\mathrm{lef}}

\nc{\Sh}{\mathrm{ST}}

\nc{\Cr}{\mathrm{Cr}}

\nc{\st}{{Schr\"oder tree}\xspace}
\nc{\sts}{{Schr\"oder trees}\xspace}

\nc{\vertset}{\Omega} 

\nc{\assop}{\quad \begin{picture}(5,5)(0,0)
\line(-1,1){10}
\put(-2.2,-2.2){$\bullet$}
\line(0,-1){10}\line(1,1){10}
\end{picture} \quad \smallskip}

\nc{\operator}{\begin{picture}(5,5)(0,0)
\line(0,-1){6}
\put(-2.6,-1.8){$\bullet$}
\line(0,1){9}
\end{picture}}

\nc{\idx}{\begin{picture}(6,6)(-3,-3)
\put(0,0){\line(0,1){6}}
\put(0,0){\line(0,-1){6}}
\end{picture}}

\nc{\pb}{{\mathrm{pb}}}
\nc{\Lf}{{\mathrm{Lf}}}

\nc{\lft}{{left tree}\xspace}
\nc{\lfts}{{left trees}\xspace}

\nc{\fat}{{fundamental averaging tree}\xspace}

\nc{\fats}{{fundamental averaging trees}\xspace}
\nc{\avt}{\mathrm{Avt}}

\nc{\rass}{{\mathit{RAss}}}

\nc{\aass}{{\mathit{AAss}}}

\nc{\vin}{{\mathrm Vin}}    
\nc{\lin}{{\mathrm Lin}}    
\nc{\inv}{\mathrm{I}n}
\nc{\gensp}{V} 
\nc{\genbas}{\mathcal{V}} 
\nc{\bvp}{V_P}     
\nc{\gop}{{\,\omega\,}}     

\nc{\bin}[2]{ (_{\stackrel{\scs{#1}}{\scs{#2}}})}  
\nc{\binc}[2]{ \left (\!\! \begin{array}{c} \scs{#1}\\
    \scs{#2} \end{array}\!\! \right )}  
\nc{\bincc}[2]{  \left ( {\scs{#1} \atop
    \vspace{-1cm}\scs{#2}} \right )}  
\nc{\bs}{\bar{S}} \nc{\cosum}{\sqsubset} \nc{\la}{\longrightarrow}
\nc{\rar}{\rightarrow} \nc{\dar}{\downarrow} \nc{\dprod}{**}
\nc{\dap}[1]{\downarrow \rlap{$\scriptstyle{#1}$}}
\nc{\md}{\mathrm{dth}} \nc{\uap}[1]{\uparrow
\rlap{$\scriptstyle{#1}$}} \nc{\defeq}{\stackrel{\rm def}{=}}
\nc{\disp}[1]{\displaystyle{#1}} \nc{\dotcup}{\
\displaystyle{\bigcup^\bullet}\ } \nc{\gzeta}{\bar{\zeta}}
\nc{\hcm}{\ \hat{,}\ } \nc{\hts}{\hat{\otimes}}
\nc{\barot}{{\otimes}} \nc{\free}[1]{\bar{#1}}
\nc{\uni}[1]{\tilde{#1}} \nc{\hcirc}{\hat{\circ}} \nc{\lleft}{[}
\nc{\lright}{]} \nc{\lc}{\lfloor} \nc{\rc}{\rfloor}
\nc{\curlyl}{\left \{ \begin{array}{c} {} \\ {} \end{array}
    \right .  \!\!\!\!\!\!\!}
\nc{\curlyr}{ \!\!\!\!\!\!\!
    \left . \begin{array}{c} {} \\ {} \end{array}
    \right \} }
\nc{\longmid}{\left | \begin{array}{c} {} \\ {} \end{array}
    \right . \!\!\!\!\!\!\!}
\nc{\onetree}{\bullet} \nc{\ora}[1]{\stackrel{#1}{\rar}}
\nc{\ola}[1]{\stackrel{#1}{\la}}
\nc{\ot}{\otimes} \nc{\mot}{{{\boxtimes\,}}}
\nc{\otm}{\overline{\boxtimes}} \nc{\sprod}{\bullet}
\nc{\scs}[1]{\scriptstyle{#1}} \nc{\mrm}[1]{{\rm #1}}
\nc{\margin}[1]{\marginpar{\rm #1}}   
\nc{\dirlim}{\displaystyle{\lim_{\longrightarrow}}\,}
\nc{\invlim}{\displaystyle{\lim_{\longleftarrow}}\,}
\nc{\mvp}{\vspace{0.3cm}} \nc{\tk}{^{(k)}} \nc{\tp}{^\prime}
\nc{\ttp}{^{\prime\prime}} \nc{\svp}{\vspace{2cm}}
\nc{\vp}{\vspace{8cm}} \nc{\proofbegin}{\noindent{\bf Proof: }}
\nc{\proofend}{$\blacksquare$ \vspace{0.3cm}}
\nc{\modg}[1]{\!<\!\!{#1}\!\!>}
\nc{\intg}[1]{F_C(#1)} \nc{\lmodg}{\!
<\!\!} \nc{\rmodg}{\!\!>\!}
\nc{\cpi}{\widehat{\Pi}}
\nc{\sha}{{\mbox{\cyr X}}}  
\nc{\shap}{{\mbox{\cyrs X}}} 
\nc{\shpr}{\diamond}    
\nc{\shp}{\ast} \nc{\shplus}{\shpr^+}
\nc{\shprc}{\shpr_c}    
\nc{\msh}{\ast} \nc{\zprod}{m_0} \nc{\oprod}{m_1}
\nc{\vep}{\epsilon} \nc{\labs}{\mid\!} \nc{\rabs}{\!\mid}
\nc{\sqmon}[1]{\langle #1\rangle}

\nc{\mmbox}[1]{\mbox{\ #1\ }} \nc{\dep}{\mrm{dep}} \nc{\fp}{\mrm{FP}}
\nc{\rchar}{\mrm{char}} \nc{\End}{\mrm{End}} \nc{\Fil}{\mrm{Fil}}
\nc{\Mor}{Mor\xspace} \nc{\gmzvs}{gMZV\xspace}
\nc{\gmzv}{gMZV\xspace} \nc{\mzv}{MZV\xspace}
\nc{\mzvs}{MZVs\xspace} \nc{\Hom}{\mrm{Hom}} \nc{\id}{\mrm{id}}
\nc{\im}{\mrm{im}} \nc{\incl}{\mrm{incl}} \nc{\map}{\mrm{Map}}
\nc{\mchar}{\rm char} \nc{\nz}{\rm NZ} \nc{\supp}{\mathrm Supp}

\nc{\Alg}{\mathbf{Alg}} \nc{\Bax}{\mathbf{Bax}} \nc{\bff}{\mathbf f}
\nc{\bfk}{{\bf k}} \nc{\bfone}{{\bf 1}} \nc{\bfx}{\mathbf x}
\nc{\bfy}{\mathbf y}
\nc{\base}[1]{\bfone^{\otimes ({#1}+1)}} 
\nc{\Cat}{\mathbf{Cat}}

\nc{\detail}{\marginpar{\bf More detail}
    \noindent{\bf Need more detail!}
    \svp}
\nc{\Int}{\mathbf{Int}} \nc{\Mon}{\mathbf{Mon}}
\nc{\rbtm}{{shuffle }} \nc{\rbto}{{Rota-Baxter }}
\nc{\remarks}{\noindent{\bf Remarks: }} \nc{\Rings}{\mathbf{Rings}}
\nc{\Sets}{\mathbf{Sets}} \nc{\wtot}{\widetilde{\odot}}
\nc{\wast}{\widetilde{\ast}} \nc{\bodot}{\bar{\odot}}
\nc{\bast}{\bar{\ast}} \nc{\hodot}[1]{\odot^{#1}}
\nc{\hast}[1]{\ast^{#1}} \nc{\mal}{\mathcal{O}}
\nc{\tet}{\tilde{\ast}} \nc{\teot}{\tilde{\odot}}
\nc{\oex}{\overline{x}} \nc{\oey}{\overline{y}}
\nc{\oez}{\overline{z}} \nc{\oef}{\overline{f}}
\nc{\oea}{\overline{a}} \nc{\oeb}{\overline{b}}
\nc{\weast}[1]{\widetilde{\ast}^{#1}}
\nc{\weodot}[1]{\widetilde{\odot}^{#1}} \nc{\hstar}[1]{\star^{#1}}
\nc{\lae}{\langle} \nc{\rae}{\rangle}
\nc{\lf}{\lfloor}
\nc{\rf}{\rfloor}

\nc{\QQ}{{\mathbb Q}}
\nc{\RR}{{\mathbb R}} \nc{\ZZ}{{\mathbb Z}}

\nc{\cala}{{\mathcal A}} \nc{\calb}{{\mathcal B}}
\nc{\calc}{{\mathcal C}}
\nc{\cald}{{\mathcal D}} \nc{\cale}{{\mathcal E}}
\nc{\calf}{{\mathcal F}} \nc{\calg}{{\mathcal G}}
\nc{\calh}{{\mathcal H}} \nc{\cali}{{\mathcal I}}
\nc{\call}{{\mathcal L}} \nc{\calm}{{\mathcal M}}
\nc{\caln}{{\mathcal N}} \nc{\calo}{{\mathcal O}}
\nc{\calp}{{\mathcal P}} \nc{\calr}{{\mathcal R}}
\nc{\cals}{{\mathcal S}} \nc{\calt}{{\mathcal T}}
\nc{\calu}{{\mathcal U}} \nc{\calw}{{\mathcal W}} \nc{\calk}{{\mathcal K}}
\nc{\calx}{{\mathcal X}} \nc{\CA}{\mathcal{A}}

\nc{\fraka}{{\mathfrak a}} \nc{\frakA}{{\mathfrak A}}
\nc{\frakb}{{\mathfrak b}} \nc{\frakB}{{\mathfrak B}}
\nc{\frakD}{{\mathfrak D}} \nc{\frakF}{\mathfrak{F}}
\nc{\frakf}{{\mathfrak f}} \nc{\frakg}{{\mathfrak g}}
\nc{\frakH}{{\mathfrak H}} \nc{\frakL}{{\mathfrak L}}
\nc{\frakM}{{\mathfrak M}} \nc{\bfrakM}{\overline{\frakM}}
\nc{\frakm}{{\mathfrak m}} \nc{\frakP}{{\mathfrak P}}
\nc{\frakN}{{\mathfrak N}} \nc{\frakp}{{\mathfrak p}}
\nc{\frakS}{{\mathfrak S}} \nc{\frakT}{\mathfrak{T}}
\nc{\frakX}{{\mathfrak X}}
\nc{\BS}{\mathbb{S
}}

\font\cyr=wncyr10 \font\cyrs=wncyr7
\nc{\li}[1]{\textcolor{red}{Li:#1}}
\nc{\yi}[1]{\textcolor{blue}{Yi: #1}}
\nc{\xing}[1]{\textcolor{purple}{Xing:#1}}
\nc{\revise}[1]{\textcolor{red}{#1}}

\nc{\ID}{{\rm I}}\nc{\lbar}[1]{\overline{#1}}\nc{\bre}{{\rm bre}}
\nc{\sd}{\cals}\nc{\rb}{\rm RB}\nc{\A}{\rm A}\nc{\LL}{\rm L}\nc{\tx}{\tilde{X}}
\nc{\col}{\Delta_{\epsilon}}\nc{\mul}{m_{\mathrm{RT}}}\nc{\ul}{u_{RT}}\nc{\epl}{\epsilon_{RT}}
\nc{\hl}{H_{RT}}\nc{\arro}[1]{#1}\nc{\px}{P_{\tx}}\nc{\pw}{P_{\mathfrak{w}}}\nc{\pl}{B^+}
\nc{\pp}{\pl}\nc{\ppp}[1]{B^+(#1)}\nc{\dw}{\diamond_{\mathfrak{w}}}\nc{\dl}{\diamond_{\rm \ell}}
\nc{\ncshaw}{\sha^{{\rm NC}}_{\mathfrak{w}}}\nc{\ncshal}{\sha^{{\rm NC}}_{{\rm \ell}}}
\nc{\ver}{\rm V}\nc{\ld}{l}\nc{\del}{\Delta_{{\rm \ell}}}\nc{\epsl}{\epsilon_{{\rm \ell}}}
\nc{\uul}{u_{{\rm \ell}}}\nc{\oneh}{\mathbf{1}}\nc{\onew}{\mathbf{1}}
\nc{\etree}{\mathbbm{{1}}}
\nc{\conc}{m_{RT}} \nc{\subq}{\bfk Q_l} \nc{\fid}{\unlhd}  \nc{\sfid}{\lhd}
\nc{\lhl}{\leq_{h,l}} \nc{\ghl}{\geq_{hl}}
\nc{\RT}{\mathrm{RT}}

\nc{\hrtb}{H_{RT}(X\sqcup\Omega)} \nc{\hrts}{H_{\mathrm{RT}}(X, \Omega)}\nc{\rts}{\mathcal{T}(X, \Omega)}\nc{\rfs}{\mathcal{F}(X, \Omega)} \nc{\counit}{\varepsilon_{\mathrm{RT}}}

\newcommand{\tun}{\begin{picture}(5,0)(-2,-1)
\put(0,0){\circle*{2}}
\end{picture}}

\newcommand{\tdeux}{\begin{picture}(7,7)(0,-1)
\put(3,0){\circle*{2}}
\put(3,0){\line(0,1){5}}
\put(3,5){\circle*{2}}
\end{picture}}

\newcommand{\ttroisun}{\begin{picture}(15,8)(-5,-1)
\put(3,0){\circle*{2}}
\put(-0.65,0){$\vee$}
\put(6,7){\circle*{2}}
\put(0,7){\circle*{2}}
\end{picture}}
\newcommand{\ttroisdeux}{\begin{picture}(5,12)(-2,-1)
\put(0,0){\circle*{2}}
\put(0,0){\line(0,1){5}}
\put(0,5){\circle*{2}}
\put(0,5){\line(0,1){5}}
\put(0,10){\circle*{2}}
\end{picture}}

\newcommand{\tquatreun}{\begin{picture}(15,12)(-5,-1)
\put(3,0){\circle*{2}}
\put(-0.65,0){$\vee$}
\put(6,7){\circle*{2}}
\put(0,7){\circle*{2}}
\put(3,7){\circle*{2}}
\put(3,0){\line(0,1){7}}
\end{picture}}
\newcommand{\tquatredeux}{\begin{picture}(15,18)(-5,-1)
\put(3,0){\circle*{2}}
\put(-0.65,0){$\vee$}
\put(6,7){\circle*{2}}
\put(0,7){\circle*{2}}
\put(0,14){\circle*{2}}
\put(0,7){\line(0,1){7}}
\end{picture}}
\newcommand{\tquatretrois}{\begin{picture}(15,18)(-5,-1)
\put(3,0){\circle*{2}}
\put(-0.65,0){$\vee$}
\put(6,7){\circle*{2}}
\put(0,7){\circle*{2}}
\put(6,14){\circle*{2}}
\put(6,7){\line(0,1){7}}
\end{picture}}
\newcommand{\tquatrequatre}{\begin{picture}(15,18)(-5,-1)
\put(3,5){\circle*{2}}
\put(-0.65,5){$\vee$}
\put(6,12){\circle*{2}}
\put(0,12){\circle*{2}}
\put(3,0){\circle*{2}}
\put(3,0){\line(0,1){5}}
\end{picture}}
\newcommand{\tquatrecinq}{\begin{picture}(9,19)(-2,-1)
\put(0,0){\circle*{2}}
\put(0,0){\line(0,1){5}}
\put(0,5){\circle*{2}}
\put(0,5){\line(0,1){5}}
\put(0,10){\circle*{2}}
\put(0,10){\line(0,1){5}}
\put(0,15){\circle*{2}}
\end{picture}}

\newcommand{\tcinqdeux}{\begin{picture}(15,14)(-5,-1)
\put(3,0){\circle*{2}}
\put(-0.65,0){$\vee$}
\put(6,7){\circle*{2}}
\put(0,7){\circle*{2}}
\put(3,7){\circle*{2}}
\put(3,0){\line(0,1){7}}
\put(0,7){\line(0,1){7}}
\put(0,14){\circle*{2}}
\end{picture}}

\newcommand{\tdun}[1]
{\begin{picture}(10,5)(-2,-1)
\put(0,0){\circle*{2}}
\put(3,-2){\tiny #1}
\end{picture}}

\newcommand{\tddeux}[2]{\begin{picture}(12,5)(0,-1)
\put(3,0){\circle*{2}}
\put(3,0){\line(0,1){5}}
\put(3,5){\circle*{2}}
\put(6,-3){\tiny #1}
\put(6,3){\tiny #2}
\end{picture}}

\newcommand{\tdtroisun}[3]{\begin{picture}(20,12)(-5,-1)
\put(3,0){\circle*{2}}
\put(-0.65,0){$\vee$}
\put(6,7){\circle*{2}}
\put(0,7){\circle*{2}}
\put(5,-2){\tiny #1}
\put(8,5){\tiny #2}
\put(-6,5){\tiny #3}
\end{picture}}

\newcommand{\tdquatredeux}[4]{\begin{picture}(20,20)(-5,-1)
\put(3,0){\circle*{2}}
\put(-.65,0){$\vee$}
\put(6,7){\circle*{2}}
\put(0,7){\circle*{2}}
\put(0,14){\circle*{2}}
\put(0,7){\line(0,1){7}}
\put(5,-2){\tiny #1}
\put(9,5){\tiny #2}
\put(-6,5){\tiny #3}
\put(-6,12){\tiny #4}
\end{picture}}
\newcommand{\tdquatretrois}[4]{\begin{picture}(20,20)(-5,-1)
\put(3,0){\circle*{2}}
\put(-.65,0){$\vee$}
\put(6,7){\circle*{2}}
\put(0,7){\circle*{2}}
\put(6,14){\circle*{2}}
\put(6,7){\line(0,1){7}}
\put(5,-2){\tiny #1}
\put(8,5){\tiny #2}
\put(-6,5){\tiny #4}
\put(8,12){\tiny #3}
\end{picture}}

\begin{document}

\title[Weighted infinitesimal unitary bialgebras on rooted forests and weighted cocycles]{Weighted infinitesimal unitary bialgebras on rooted forests and weighted cocycles}
%
\author{Yi Zhang}
\address{School of Mathematics and Statistics, Lanzhou University, Lanzhou, Gansu 730000, P.\,R. China}
         \email{zhangy2016@lzu.edu.cn}

\author{Dan Chen}
\address{School of Mathematics and Statistics, Lanzhou University, Lanzhou, Gansu 730000, P.\,R. China}
         \email{chendan@lzu.edu.cn}

\author{Xing Gao$^{*}$}
\footnotetext{* Corresponding author.}
\address{School of Mathematics and Statistics, Key Laboratory of Applied Mathematics and Complex Systems, Lanzhou University, Lanzhou, Gansu 730000, P.\,R. China}
         \email{gaoxing@lzu.edu.cn}

\author{Yan-Feng Luo} \address{School of Mathematics and Statistics,
Key Laboratory of Applied Mathematics and Complex Systems,
Lanzhou University, Lanzhou, Gansu 730000, P.\,R. China}
         \email{luoyf@lzu.edu.cn}

\date{\today}
\begin{abstract}
In this paper, we define a new coproduct on the space of decorated planar rooted forests to equip it with a weighted infinitesimal unitary bialgebraic structure.
We introduce the concept of $\Omega$-cocycle infinitesimal bialgebras of weight $\lambda$ and then prove that the space of decorated planar rooted forests $H_{\mathrm{RT}}(X,\Omega)$, together with a set of grafting operations $\{ B^+_\omega \mid \omega\in \Omega\}$,
is the free $\Omega$-cocycle infinitesimal unitary bialgebra of weight $\lambda$
on a set $X$, involving a weighted version of a Hochschild 1-cocycle condition.
As an application, we equip a free cocycle infinitesimal unitary bialgebraic
structure on the undecorated planar rooted forests,
which is the object studied in the well-known (noncommutative) Connes-Kreimer Hopf algebra.
\end{abstract}

\subjclass[2010]{
16W99, 
16S10, 
16T10, 
16T30,  
81R10,  
}

\keywords{Rooted forest; Infinitesimal bialgebra; Cocycle condition; Operated algebra}

\maketitle

\tableofcontents

\setcounter{section}{0}

\allowdisplaybreaks

\section{Introduction}
An infinitesimal bialgebra is a module $A$ which is simultaneously an algebra (possibly without a unit) and a coalgebra (possibly without a counit) such that the coproduct $\Delta$ is a derivation of $A$ in the sense that
\begin{equation}
\Delta(ab)=a\cdot\Delta(b)+\Delta(a)\cdot b\quad \text{ for } a, b\in A.
\mlabel{eq:comp1}
\end{equation}
When an infinitesimal bialgebra has an antipode, it will be called an infinitesimal Hopf algebra.
Infinitesimal bialgebras, first introduced by Joni and Rota~\mcite{JR}, are in order to give an
algebraic framework for the calculus of Newton divided differences. The basic theory of infinitesimal bialgebras and infinitesimal Hopf algebras was developed in~\mcite{MA, Agu01, Aguu02}.
Furthermore, infinitesimal bialgebras are also closely related to associative Yang-Baxter equations, Drinfeld's doubles, pre-Lie algebras and Drinfeld's Lie bialgebras~\mcite{Agu01}.
Recently, Wang~\mcite{WW14} generalized  Aguiar's results by studying the Drinfeld's double for braided infinitesimal Hopf algebras in Yetter-Drinfeld categories.
Another different version of infinitesimal bialgebras and infinitesimal Hopf algebras was defined by Loday and Ronco~\mcite{LR06}
and further studied by Foissy~\mcite{Foi09, Foi10}, in the sense that
\begin{equation}
\Delta(ab)=a\cdot\Delta(b)+\Delta(a)\cdot b-a\ot b\quad\text{ for } a, b\in A.
\mlabel{eq:comp2}
\end{equation}

In 2010, the relationship between classical rime solutions of the Yang-Baxter equation and B$\acute{\mathrm{e}}$zout operators was investigated by Ogievetsky and Popov~\mcite{OP10}, who turned the associative Yang-Baxter equation~\mcite{BGN1, BGN2} into a general structure, called non-homogenous associative classical Yang-Baxter equation.
Surprisingly, in the spirit of the well-known fact that a solution of the associative Yang-Baxter equation gives an infinitesimal bialgebra~\mcite{MA},
Ogievetsky and Popov~\mcite{OP10} clarified an algebraic meaning of the non-homogenous associative classical Yang-Baxter equation,
involving a coproduct given by
\begin{equation}
\Delta_{r}(a) :=a\cdot r- r\cdot a-\lambda (a\ot 1) \quad \text{ for }\, a\in A.
\mlabel{eq:id1}
\end{equation}
Here $\lambda\in \bfk$ and $r\in A\ot A$ is a solution of the non-homogenous associative classical Yang-Baxter equation.
Note that~\mcite{OP10} Eq.~(\mref{eq:id1}) satisfies
\begin{equation}
\Delta_{r}(ab)=a\cdot\Delta_{r}(b)+\Delta_{r}(a)\cdot b+\lambda (a\ot b)\quad \text{ for } a, b\in A,
\mlabel{eq:id2}
\end{equation}
which is precisely a uniform version of  the two compatibilities--Eqs~(\mref{eq:comp1}) and~(\mref{eq:comp2}).
Such an algebraic structure was called an infinitesimal unitary bialgebra of weight $\lambda$ in~\mcite{GZ, ZGZ18}.
We would like to point out that weighted infinitesimal unitary bialgebras have a close connection with pre-Lie algebras.
For example, Aguiar~\mcite{Aguu02} constructed a pre-Lie algebra from an infinitesimal bialgebra of weight zero. Motivated by Aguiar's construction,  a pre-Lie algebra from a weighted infinitesimal unitary bialgebra was derived in~\mcite{GZ}.

The rooted forest is a significant object studied in algebra and combinatorics.
One of the most important examples is the Connes-Kreimer Hopf algebra of rooted forests, which was introduced and studied extensively
in~\mcite{CK98, GL89, Hof03, Kre98, Moe01}.
In particular, the Connes-Kreimer Hopf algebra serves as a ``baby model" of Feynmann diagrams in the algebraic approach of
the renormalization in quantum field theory~\mcite{BF03, BF10, CK1, GPZ11, GPZ1, GZ08}.
It is also related to many other Hopf algebras built on rooted forests, such as Foissy-Holtkamp~\mcite{Foi02,Fois02,Hol03}, Grossman-Larson~\mcite{GL89} and Loday-Ronco~\mcite{LR98}.
One reason for the significance of these algebraic structures on rooted forests is that most of them possess universal properties, involving a Hochschild 1-cocycle, which have interesting applications in renormalization.
For example, the Connes-Kreimer Hopf algebra of rooted forests
inherits its algebra structure from the initial object in the category of (commutative) algebras with a linear operator~\mcite{Foi02,Moe01}.
Recently this universal property of rooted forests was generalized in~\mcite{ZGG16} in terms of decorated planar rooted forests,
and the universal property of Loday-Ronco Hopf algebra was investigated in~\mcite{ZG18} in terms of decorated planar binary trees.

The concept of algebras with (one or more) linear operators  was first introduced by Kurosh~\mcite{Kur60} but forgotten until it was rediscovered by Guo~\mcite{Guo09}, who constructed the free objects of such algebras in terms of various combinatorial objects, such as Motzkin paths, rooted forests and bracketed words. There such structure was called an $\Omega$-operated algebra, where $\Omega$ is a set to index the linear operators.
See also~\cite{BCQ10, GG, Gub, GSZ}.
It has been observed that the decorated planar rooted forests $H_{\mathrm{RT}}(\Omega)$ whose vertices are decorated by a set $\Omega$,
together with a set of grafting operations $\{B^+_\omega\mid \omega\in \Omega\})$, is a free object on the empty set in  the category of $\Omega$-operated algebras~\mcite{KP, ZGG16}.
Particularly, the noncommutative Connes-Kreimer Hopf algebra $H_{\mathrm{RT}}$ of planar rooted forests equipped with the grafting operation $B^+$ is a free operated algebra~\mcite{Guo09}.

As a related result, an infinitesimal unitary bialgebra of weight zero on rooted forests has been established in~\mcite{GW}. Using an infinitesimal version of the Hochschild 1-cocycle condition, they showed that  the space of decorated planar rooted forests is the free cocycle infinitesimal unitary bialgebra of weight zero. However, this infinitesimal 1-cocycle condition is not a real Hochschild 1-cocycle condition.
It is almost a natural question to wonder whether we can construct an infinitesimal (unitary) bialgebra of weight $\lambda$
on decorated rooted forests, by using a Hochschild 1-cocycle condition.
The present paper gives a positive answer to this question.
Namely, we first propose the concept of weighted $\Omega$-cocycle infinitesimal unitary bialgebras,
involving a weighted version of a Hochschild 1-cocycle condition.
Then we prove that the space of decorated planar rooted forests $H_{\mathrm{RT}}(X, \Omega)$ is the free objects in this category,
provided suitable operations are equipped. This freeness characterization of decorated planar rooted forests gives an algebraic explanation of the fundamental roles played by these combinatorial objects.

{\bf Structure of the Paper.}
In Section~\mref{sec:ibw}, we recall the concept of a weighted infinitesimal (unitary) bialgebra and show that some well-known algebras possess a weighted infinitesimal (unitary) bialgebra.

In Section~\mref{sec:infbi}, after summarizing concepts and basic facts on decorated rooted forests, we construct a new coproduct by a weighted version of a Hochschild 1-cocycle condition (Eq.~(\mref{eq:cdbp})) on decorated planar rooted forests $H_{\mathrm{RT}}(X,\Omega)$ to equip it with a new coalgebra structrue (Theorem~\mref{thm:rt1}).
Further $H_{\mathrm{RT}}(X,\Omega)$ can be turned into an infinitesimal unitary bialgebra of weight $\lambda$ with respect to the concatenation product and the empty tree as its unit (Theorem~\mref{thm:rt2}).

In Section~\mref{sec:ope}, under the framework of operated algebras,
we propose the concept of weighted $\Omega$-cocycle infinitesimal unitary bialgebras~(Definition~\mref{defn:xcobi}~(\mref{it:def3})), involving a weighted 1-cocycle condition. Having this concept in hand, we prove that $H_{\RT}(X,\Omega)$ is the free $\Omega$-cocycle infinitesimal unitary bialgebra of weight $\lambda$ on a set $X$ (Theorem~\mref{thm:propm}). As an application, we obtain that the undecorated planar rooted forests is the free cocycle infinitesimal unitary bialgebra of weight $\lambda$ on the empty set (Corollary~\mref{coro:rt16}).


{\bf Notation.}
Throughout this paper, let $\bfk$ be a unitary commutative ring unless the contrary is specified,
which will be the base ring of all modules, algebras, coalgebras, bialgebras, tensor products, as well as linear maps.
By an algebra we mean an associative algebra (possibly without unit)
and by a coalgebra we mean a coassociative coalgebra (possibly without counit).
We use the Sweedler notation:$$\Delta(a) = \sum_{(a)} a_{(1)} \ot a_{(2)}.$$
For a set $Y$, denote by $M(Y)$ and $S(Y)$ the free monoid and semigroup on $Y$, respectively.
For an algebra $A$, $A\ot A$ is viewed as an $(A,A)$-bimodule in the standard way
\begin{equation}
a\cdot(b\otimes c):=ab\otimes c\,\text{ and }\, (b\otimes c)\cdot a:= b\otimes ca,
\mlabel{eq:dota}
\end{equation}
where $a,b,c\in A$.

\section{Weighted infinitesimal  unitary bialgebras and examples}\label{sec:ibw}
In this section, we recall the concept of a weighted infinitesimal unitary bialgebra~\mcite{GZ, OP10}, which generalize simultaneously the one introduced by Joni and Rota~\mcite{JR} and the one initiated by Loday and Ronco~\mcite{LR06}.
Based on the mixture of Eqs.~(\mref{eq:comp1}) and~(\mref{eq:comp2}) into Eq.~(\mref{eq:id2}) by Ogievetsky and Popov~\mcite{OP10}, we
propose

\begin{defn}\mcite{GZ}\mlabel{def:iub}
Let $\lambda$ be a given element of $\bfk$.
An {\bf infinitesimal bialgebra} (abbreviated {\bf $\epsilon$-bialgebra}) {\bf of weight $\lambda$} is a triple $(A,m,\Delta)$
consisting of an algebra $(A,m)$ (possibly without a unit) and a coalgebra $(A,\Delta)$ (possibly without a counit) that satisfies
\begin{equation}
\Delta (ab)=a\cdot \Delta(b)+\Delta(a) \cdot b+\lambda (a\ot b)\quad \text{ for } a, b\in A.
\mlabel{eq:cocycle}
\end{equation}
If further $(A,m,1)$ is a unitary algebra, then the quadruple $(A,m,1, \Delta)$ is called an {\bf infinitesimal unitary bialgebra} {\bf of weight $\lambda$}.
\end{defn}

\begin{defn}\mcite{GZ}
Let $A$ and  $B$ be two $\epsilon$-bialgebras of weight $\lambda$.
A map $\phi : A\rightarrow B$ is called an {\bf infinitesimal bialgebra morphism} if $\phi$ is an algebra morphism and a coalgebra morphism.
\end{defn}

We shall use the infix notation $\epsilon$- interchangeably with the adjective ``infinitesimal" throughout the rest of this paper.

\begin{remark}\label{remk:4rem}
\begin{enumerate}
\item \label{remk:units}Let $(A,m,1, \Delta)$ be an $\epsilon$-bialgebra of weight $\lambda$. Then $\Delta(1)=-\lambda(1\ot1)$, as
    \begin{align*}
    \Delta(1)=\Delta(1\cdot1)=1 \cdot \Delta(1) +\Delta(1)\cdot 1 +\lambda (1\ot 1)=2\Delta(1)+\lambda (1\ot 1).
    \end{align*}
\item  The $\epsilon$-bialgebra introduced by Joni and Rota~\mcite{JR} is the $\epsilon$-bialgebra of weight zero,
and the $\epsilon$-bialgebra originated from Loday and Ronco~\mcite{LR06} is the $\epsilon$-bialgebra of weight $-1$.

\item Twenty years after Joni and Rota~\mcite{JR}, Aguiar~\mcite{MA} introduced the concept of an $\epsilon$-Hopf algebra
and pointed out that there is no non-zero $\epsilon$-bialgebra which is both unitary and counitary when $\lambda=0$. Indeed, it follows the counicity that
$$1\ot 1_{\bfk}=(\id \ot \epsilon)\Delta(1)=0$$ and so $1=0$.

\item Let $(A, \mu, \Delta)$ be an $\epsilon$-unitary bialgebra of weight $\lambda$. Denote by
\begin{align}
\rhd: A\ot A \to A, \, a\ot b \mapsto a\rhd b:=\sum_{(b)}b_{(1)}a b_{(2)},
\mlabel{eq:preope}
\end{align}
where $b_{(1)}$ and $b_{(2)}$ are from the Sweedler notation $\Delta (b)=\sum_{(b)}b_{(1)}\ot b_{(2)}$.
Then $A$ equipped with the $\rhd$ in Eq.~(\mref{eq:preope}) is a pre-Lie algebra~\cite{GZ}.
\end{enumerate}
\end{remark}

Some well-known algebras possess weighted infinitesimal bialgebraic structures, via constructions of suitable coproducts.

\begin{exam}\label{exam:bialgebras}
Here are some examples of weighted $\epsilon$-bialgebras.
\begin{enumerate}
\item Any  algebra $( A, m)$ is an $\epsilon$-bialgebra of weight zero when the coproduct is taken to be $\Delta=0$.

\item \cite[Example~2.3.2]{MA} A quiver $Q=(Q_0,Q_1, s,t)$ is a quadruple consisting of a set $Q_0$ of vertices, a set $Q_1$ of arrows, and two maps $s,t :Q_1\rightarrow Q_0$ which associate  each arrow $a\in Q_1$ to its source $s(a)\in Q_0$ and its target $t(a)\in Q_0$.  The path algebra $\bfk Q$ can be turned into an $\epsilon$-unitary bialgebra of weight zero with the coproduct $\Delta$ given by:
    $$\Delta(a_1\cdots a_n):= \left\{
    \begin{array}{ll}
    0 & \text{ if } n = 0,\\
   s(a_1)\ot t(a_1)& \text{ if } n=1,\\
    s(a_{1})\ot a_2 \cdot\cdots a_n+a_1\cdots a_{n-1}\ot t(a_n) \\
+\sum_{i=1}^{n-2}a_1\cdots a_i\ot a_{i+2}\cdots a_n & \text{ if } n\geq 2,
    \end{array}
    \right.$$
 where $a_1\cdots a_n$ is a path in $\bfk Q$. Here we use the convention that $a_1\cdots a_n\in Q_0$ when $n=0$.
 \item \cite[Section~1.4]{Foi08} Let $(A, m, 1,\Delta_c, \varepsilon, c)$ be a braided bialgebra with $A = \bfk\oplus \ker\varepsilon$ and the braiding
$c:A\ot A \to A\ot A$ given by
\begin{align*}
c:\left\{
    \begin{array}{ll}
    1\ot 1 \mapsto 1\ot 1,\\
   a\ot 1\mapsto 1\ot a, \\
    1\ot b \mapsto b\ot 1,\\
    a\ot b \mapsto 0, \quad \text{ where } a,b\in \ker \varepsilon.
    \end{array}
    \right.
\end{align*}
Then $(A, m, 1,\Delta_c)$ is an $\epsilon$-unitary bialgebra of weight $-1$.
\end{enumerate}
\end{exam}

\section{Weighted infinitesimal unitary bialgebras of decorated planar rooted forests}
\label{sec:infbi}
In this section, we first show a general way to decorate planar rooted forests that generalizes the constructions of decorated rooted forests introduced and studied in~\mcite{Foi02, Guo09, Sta97}. Using a weighted 1-cocycle condition, we then define a coproduct on the space of new decorated planar rooted forests to equip it with a coalgebraic structure, with an eye toward constructing a weighted infinitesimal unitary bialgebra on it.

\subsection{New decorated planar rooted forests}
A $\mathbf{rooted\ tree}$ is a finite graph, connected and without cycles, with a distinguished vertex called the $\mathbf{root}$. A $\mathbf{planar\ rooted\ tree}$  is a rooted tree with a fixed embedding into the plane.
The first few planar rooted trees are listed below:
$$\tun,\ \tdeux,\ \ttroisun,\ \tquatreun,\ \ttroisdeux,\  \tquatretrois,\ \tquatredeux,\ \tquatrequatre,\ \tcinqdeux,\ \tquatrecinq,$$
where the root of a tree is on the bottom.
Let $\calt$ denote the set of planar rooted trees and $M(\calt)$ the free monoid generated by $\calt$ with the concatenation product, denoted by $\mul$ and usually suppressed. The empty tree in $M(\calt)$ is denoted by $\etree$.
A {\bf planar rooted forest} is a noncommutative concatenation of planar rooted trees, denoted by $F=T_1\cdots T_n$ with $T_1, \ldots, T_n \in \calt$. Here we use the convention that $F=\etree$ when $n=0$. The first few planar rooted forests are listed below:
$$\etree,\ \,\tun,\ \,\tun\tun,\ \,\tdeux,\ \, \tdeux\tun,\ \,\tun \tdeux,\ \, \tun\ttroisun,\ \,\tdeux\tun\tun,\ \,\tun \tdeux \tun,\ \,\ttroisun\ttroisdeux.$$

Let $\Omega$ be a nonempty set, and let $X$ be a set whose elements are not in the set $\Omega$.
For the nonempty set $X\sqcup \Omega$, let
$\calt(X\sqcup \Omega)$ (resp.~$\calf(X\sqcup \Omega):=M(\calt(X\sqcup \Omega))$) denote the set of planar rooted trees (resp.~forests)
whose vertices (leaf vertices and internal vertices) are decorated by elements of $X\sqcup \Omega$. Define $H_{\mathrm{RT}}(X\sqcup \Omega):=\bfk \calf(X\sqcup \Omega)$ to be the free $\bfk$-module spanned by $\calf(X\sqcup \Omega)$.

Let $\rts$ (resp.~$\rfs$) denote the subset of $\calt(X\sqcup \Omega)$ (resp.~$\calf(X\sqcup \Omega)$) consisting of vertex decorated planar rooted trees (resp. forests) whose internal vertices are decorated by elements of $\Omega$ exclusively and  leaf vertices are decorated by elements of $X\sqcup \Omega$. In other words, all internal vertices, as well as possibly some of the leaf vertices, are decorated by $\Omega$. The only vertex of the tree $\bullet$ is taken to be a leaf vertex.
The following are some decorated planar rooted trees in $\rts$:
$$\tdun{$\alpha$},\ \, \tdun{$x$},\ \, \tddeux{$\alpha$}{$\beta$},\ \,  \tddeux{$\alpha$}{$x$}, \ \, \tdtroisun{$\alpha$}{$\beta$}{$\gamma$},\ \,\tdtroisun{$\alpha$}{$x$}{$\gamma$}, \ \,\tdtroisun{$\alpha$}{$x$}{$y$}, \ \, \tdquatretrois{$\alpha$}{$\beta$}{$\gamma$}{$\beta$},\ \, \tdquatretrois{$\alpha$}{$\beta$}{$\gamma$}{$x$}, \ \, \tdquatretrois{$\alpha$}{$\beta$}{$x$}{$y$},$$
with $\alpha,\beta,\gamma\in \Omega$ and $x, y \in X$.
Define
\begin{align*}
\hrts:= \bfk \rfs=\bfk M(\rts)
\end{align*}
to be the free $\bfk$-module spanned by $\rfs$.
For each  $\omega\in \Omega$, define
$$B^+_\omega:\hrts\to \hrts$$
to be the linear grafting operation by taking $\etree$ to $\bullet_\omega$ and sending a rooted forest in $\hrts$ to its grafting with the new root decorated by $\omega$. For example,
$$B_{\omega}^{+}(\etree)=\tdun{$\omega$} \ ,\ \
 B_{\omega}^{+}(\tdun{$x$} \tdun{$y$})=\tdtroisun{$\omega$}{$y$}{$x$}\ , \ \
B_{\omega}^{+}(\tdun{$x$}\tddeux{$\alpha$}{$y$})=\tdquatretrois{$\omega$}{$\alpha$}{$y$}{$x$},\ \  B_{\omega}^{+}(\tddeux{$\beta$}{$\alpha$}\tdun{$x$})=
\tdquatredeux{$\omega$}{$x$}{$\beta$}{$\alpha$},$$
where $\alpha, \beta, \omega\in \Omega$ and $x, y \in X$. Note that $\hrts$ is closed under the concatenation $\mul$.

\begin{remark}\label{re:3ex}
Here are some special cases of our decorated planar rooted forests.
\begin{enumerate}
\item If $X=\emptyset$ and $\Omega$ is a singleton set, then all decorated planar rooted forests in $\mathcal{F}(X, \Omega)$ have the same decoration, which is the object studied in the well-known Foissy-Holtkamp Hopf algebra---the noncommutative version of Connes-Kreimer Hopf algebra~\mcite{Foi02, Hol03}.

\item If $X = \emptyset$, then $\mathcal{F}(X, \Omega)$ was studied by Foissy~\mcite{Foi02,Fois02}, in which a decorated noncommutative version of Connes-Kreimer Hopf algebra was constructed. \mlabel{it:2ex}

\item If $\Omega$ is a singleton set, then $\rfs$ was introduced and studied in~\mcite{ZGG16} to construct a cocycle Hopf algebra on decorated planar rooted forests.
\item The rooted forests in $\rfs$ with leaf vertices decorated by elements of $X$ and internal vertices decorated by elements of $\Omega$ were introduced in~\mcite{Guo09}. However, this decoration can't deal with the unity and the algebraic structures on this decorated rooted forests are all nonunitary. The distinction between unitary and nonunitary for $\rfs$ is more significant than for an associative algebra, because of the involvement of the grafting operation.
\end{enumerate}
\end{remark}

The following are two basic definitions that will be used in the remainder of the paper. See \mcite{GZ, Guo09} for detailed discussions.
For $F=T_1\cdots T_{n}\in \rfs$ with $n\geq 0$ and $T_1,\cdots,T_{n}\in \rts$, we define $\bre(F):=n$
to be the {\bf breadth} of $F$. Here we use the convention that $\bre(\etree) = 0$ when $n=0$.
In order to define the depth of a decorated planar rooted forests, we build a recursive structure on $\rfs$.
Define $\bullet_{X}:=\{\bullet_{x}\mid x\in X\}$ and set
\begin{align*}
\calf_0:=M(\bullet_{X})=S(\bullet_{X})\sqcup \{\etree\},
\end{align*}
where $M(\bullet_{X})$ (resp.~$S(\bullet_{X})$) is the submonoid (resp.~subsemigroup) of $\rfs$ generated by $\bullet_{X}$.
Here we are abusing notion slightly since $M(\bullet_{X})$ (resp.~$S(\bullet_{X})$) is also isomorphic to the free monoid
(resp. semigroup) generated by $\bullet_{X}$.
Suppose that $\calf_n$ has been defined for an $n\geq 0$, then define
\begin{align*}
\calf_{n+1}:=M(\bullet_{X}\sqcup (\sqcup_{\omega\in \Omega} B_{\omega}^{+}(\calf_n))).
\end{align*}
Thus we obtain $\calf_n\subseteq \calf_{n+1}$ and can define
\begin{align}
\rfs := \lim_{\longrightarrow} \calf_n=\bigcup_{n=0}^{\infty}\calf_n.
\mlabel{eq:rdeff}
\end{align}
Now elements $F\in \calf_n\setminus \calf_{n-1}$ are said to have {\bf depth} $n$, denoted by $\dep(F) = n$.
For example,
\begin{align*}
\dep(\etree) =&\ \dep(\bullet_x) =0,\ \dep(\bullet_\omega)=\dep(B^+_{\omega}(\etree)) = 1,\  \dep(\tddeux{$\omega$}{$y$})=
\dep(B^+_{\omega}(\bullet_y)) =1, \\
\dep&(\tdun{$x$}\tddeux{$\omega$}{$y$})=
\dep(\tdun{$x$}B^+_{\omega}(\bullet_y)) =1,\, \dep(~\tdtroisun{$\omega$}{$x$}{$\alpha$}) = \dep(B^+_{\omega}(B^+_{\alpha}(\etree) \bullet_x)) = 2,
\end{align*}
where $\alpha,\omega\in \Omega$ and $x, y \in X$.

\subsection{Cartier-Quillen cohomology}
Given an algebra $A$  and a bimodule $M$ over $A$. Let $H^{*}(A, M)$ denote the {\bf Hochschild cohomology} of $A$ with coefficients in $M$ which was defined from a complex with maps $A^{\ot n}\rightarrow M$ as cochains, see~\mcite{Lod92} for more details. Let $(C, \Delta)$ be a coalgebra and $(B, \delta_{G}, \delta_{D})$ be a bicomodule over $C$. The {\bf Cartier-Quillen cohomology} of $C$ with coefficients in $B$ is a dual notation of the Hochschild cohomology. Explicitly, it is a cohomology of the complex $\mathrm{Hom}_{\mathbf{k}}(B, C^{\ot n})$ with the maps $b_n: \mathrm{Hom}_{\mathbf{k}}(B, C^{\ot n})\rightarrow \mathrm{Hom}_{\mathbf{k}}(B, C^{\ot (n+1)})$ given by
\begin{align*}
b_n(L)=(\id \ot L)\circ \delta_{G}+\sum_{i=1}^{n}(-1)^i(\id _{C}^{\ot{(i-1)}}\ot \Delta \ot \id _{C}^{\ot{(n-i)}})L+(-1)^{n+1}(L\ot \id)\circ \delta_D,
\end{align*}
where $L: B\rightarrow C^{\ot n}$. In particular,  a linear map $L:B\rightarrow C$ is the {\bf $1$-cocycle} for this cohomology precisely when it satisfies the following condition:
\begin{align*}
\Delta\circ L= (L\ot \id)\circ \delta_D + (\id \ot L)\circ \delta_{G},
\end{align*}
see~\mcite{Fo3, Moe01} for more details. Let $e$ be a {\bf group-like element of weight $\lambda$} of $C$, that is, $\Delta(e)=\lambda (e\ot e)$. We consider the bicomodule $(C, \Delta, \delta_D)$ with $\delta_{D}(x)=\lambda (x\ot e)$ for any $x\in C$. Then the $1$-cocycle $L$
is a linear endomorphism of $C$ satisfying
\begin{align}
\Delta \circ L(x)= L(x)\ot \lambda e+(\id \ot L)\circ \Delta(x) \quad \text{ for } x\in C.
\mlabel{eq:wcocycle}
\end{align}
We call Eq.~(\mref{eq:wcocycle}) the {\bf 1-cocycle condition of weight $\lambda$}.

\begin{remark}
\begin{enumerate}
\item The group like elements in infinitesimal unitary bialgebras always exist. Indeed, the unit of an infinitesimal unitary bialgebra is a group like element of weight $-\lambda$.

\item When $L=B^+$ and $\lambda=1$, the weighted 1-cocycle condition in~Eq.~(\mref{eq:wcocycle}) is
\begin{equation}
\Delta_{\RT} (F)=\Delta_{\RT} B^{+}(\lbar{F})= F \otimes \etree + (\id\otimes B^{+})\Delta_{\RT}(\lbar{F})
\quad \text{ for } F=B^{+}(\lbar{F})\in \mathcal{F},
\mlabel{eq:usuco}
\end{equation}
which is the usual 1-cocycle condition employed in~\mcite{CK98,Fo3,ZGG16}.
Here the empty tree $\etree$ is the unique group like element in the Connes-Kreimer Hopf algebra.
\end{enumerate}
\end{remark}

\subsection{Weighted infinitesimal unitary bialgebras on decorated planar rooted forests}
In this subsection, we shall equip a weighted infinitesimal unitary bialgebraic structure on decorated planar rooted forests.

Let us define a new coproduct $\col$ on $\hrts$ by induction on depth.
By linearity, we only need to define $\col(F)$ for basis elements $F\in \rfs$.
For the initial step of $\dep(F)=0$, we  define
\begin{equation}
\col(F) :=
\left\{
\begin{array}{ll}
-\lambda (\etree \ot \etree) & \text{ if } F = \etree, \\
\bullet_{x}\ot  \bullet_{x}& \text{ if } F = \bullet_x \text{ for some } x \in X,\\
\bullet_{x_{1}}\cdot \col(\bullet_{x_{2}}\cdots\bullet_{x_{m}})+ \col(\bullet_{x_{1}}) \cdot (\bullet_{x_{2}}\cdots\bullet_{x_{m}})\\
+\lambda\bullet_{x_{1}}\ot \bullet_{x_{2}}\cdots\bullet_{x_{m}}
& \text{ if }  F=\bullet_{x_{1}}\cdots \bullet_{x_{m}} \text{ with } m\geq 2.
\end{array}
\right.
 \mlabel{eq:dele}
\end{equation}
Here in the third case, the definition of $\col$ reduces to the induction on breadth and the dot action is defined in Eq.~(\mref{eq:dota}).

For the induction step of $\dep(F)\geq 1$, we reduce the definition to the induction on breadth.
If $\bre(F) = 1$, we
write $F=B_{\omega}^{+}(\lbar{F})$ for some $\omega\in \Omega$ and $\lbar{F}\in \rfs$, and define
\begin{equation}
\col(F)=\col B_{\omega}^{+}(\lbar{F}) :=  F \otimes (-\lambda\etree) + (\id\otimes B_{\omega}^{+})\col(\lbar{F}).
\mlabel{eq:dbp}
\end{equation}
In other words
\begin{align}
\col B_{\omega}^{+}=  B_{\omega}^{+} \otimes (-\lambda\etree) + (\id\otimes B_{\omega}^{+})\col.
\mlabel{eq:cdbp}
\end{align}
If $\bre(F) \geq 2$, we write $F=T_{1}T_{2}\cdots T_{m}$ with $m\geq 2$ and $T_1, \ldots, T_m \in \rts$, and define
\begin{equation}
\col(F)=T_{1}\cdot \col(T_{2}\cdots T_{m})+\col(T_{1})\cdot (T_{2}\cdots T_{m})+\lambda T_{1} \ot T_{2}\cdots T_{m}.
\mlabel{eq:dele1}
\end{equation}

\begin{remark}
By Remark~\mref{remk:4rem}~(\mref{remk:units}), the empty tree $\etree\in \hrts$  is a group like element of weight $-\lambda$,
and so Eq.~(\mref{eq:cdbp}) is the 1-cocycle condition of weight $-\lambda$.
\end{remark}

\begin{exam}\mlabel{exam:cop}
Let $x,y \in X$ and $\alpha, \beta, \gamma \in \Omega$. Then
\begin{align*}
\col(\tdun{$\alpha$})&=\col(B_{\alpha}^{+}(\etree))=-\lambda(\etree\ot \tdun{$\alpha$}+\tdun{$\alpha$}\ot \etree), \\
\col(\tddeux{$\alpha$}{$x$})&=\col(B_{\alpha}^{+}(\tdun{$x$}))=-\lambda(\tddeux{$\alpha$}{$x$}\ot \etree) +\tdun{$x$}\ot \tddeux{$\alpha$}{$x$},\\
\col(\tddeux{$\alpha$}{$\beta$})&=\col(B_{\alpha}^{+}(\tdun{$\beta$}))=-\lambda(\etree\ot \tddeux{$\alpha$}{$\beta$} +\tdun{$\beta$}\ot\tdun{$\alpha$}+\tddeux{$\alpha$}{$\beta$}\ot \etree),\\
%
\col(\tdun{$\gamma$}\tdun{$y$})&=-\lambda(\etree \ot \tdun{$\gamma$}\tdun{$y$})+\tdun{$\gamma$}\tdun{$y$}\ot \tdun{$y$},\\
\col(\tddeux{$\alpha$}{$x$}\tdun{$y$})&=\tddeux{$\alpha$}{$x$}\tdun{$y$}\ot \tdun{$y$}+\tdun{$x$}\ot \tddeux{$\alpha$}{$x$}\tdun{$y$},\\
\col(\tdun{$x$}\tddeux{$\beta$}{$\alpha$})&=-\lambda (\tdun{$x$}\tddeux{$\beta$}{$\alpha$}\ot \etree+ \tdun{$x$}\tdun{$\alpha$}\ot \tdun{$\beta$})+\tdun{$x$}\ot \tdun{$x$}\tddeux{$\beta$}{$\alpha$} ,\\
\col(~ \tdtroisun{$\alpha$}{$x$}{$\beta$})&=\col(B_{\alpha}^{+}(\tdun{$\beta$} \tdun{$x$}))=-\lambda (\etree\ot \tdtroisun{$\alpha$}{$x$}{$\beta$}+\tdtroisun{$\alpha$}{$x$}{$\beta$}\ot \etree)+\tdun{$\beta$}\tdun{$x$}\ot \tddeux{$\alpha$}{$x$}.
\end{align*}
\end{exam}
To show $(\hrts, \col)$ is a coalgebra, we need the following two lemmas.

\begin{lemma}\label{lem:rt11}
Let $\bullet_{x_{1}}\cdots \bullet_{x_{m}}\in \hrts$ with $m\geq 1$ and $x_1, \ldots, x_m\in X$. Then
$$\col(\bullet_{x_{1}}\cdots \bullet_{x_{m}})=\sum_{i=1}^{m}\bullet_{x_{1}}\cdots\bullet_{x_{i}}\otimes\bullet_{x_{i}}\cdots\bullet_{x_{m}}
+\lambda\sum_{i=1}^{m-1}\bullet_{x_{1}}\cdots\bullet_{x_{i}}\otimes\bullet_{x_{i+1}}\cdots\bullet_{x_{m}}.$$
\end{lemma}

\begin{proof}
We prove this result by induction on $m\geq 1$. For the initial step of $m=1$, we have
$$\col(\bullet_{x_{1}})=\bullet_{x_{1}} \ot \bullet_{x_{1}} ,$$
and the result is true trivially. For the induction step of $m\geq 2$, we get
\allowdisplaybreaks{
\begin{align*}
&\col(\bullet_{x_{1}}\cdots \bullet_{x_{m}}) = \bullet_{x_{1}} \cdot \col(\bullet_{x_{2}}\cdots\bullet_{x_{m}})+ \col(\bullet_{x_{1}})\cdot (\bullet_{x_{2}}\cdots\bullet_{x_{m}})+\lambda\bullet_{x_{1}}\ot \bullet_{x_{2}}\cdots\bullet_{x_{m}}  \quad (\text{by Eq.~(\ref{eq:dele})})\\
=& \bullet_{x_{1}} \cdot \col(\bullet_{x_{2}}\cdots\bullet_{x_{m}})+ (\bullet_{x_{1}}\ot \bullet_{x_{1}})\cdot (\bullet_{x_{2}}\cdots\bullet_{x_{m}})+\lambda\bullet_{x_{1}}\ot \bullet_{x_{2}}\cdots\bullet_{x_{m}} \quad (\text{by Eq.~(\ref{eq:dele})})\\
=&\bullet_{x_{1}} \cdot \col(\bullet_{x_{2}}\cdots\bullet_{x_{m}}) +\bullet_{x_{1}} \otimes\bullet_{x_{1}}\bullet_{x_{2}} \cdots\bullet_{x_{m}}+\lambda\bullet_{x_{1}}\ot \bullet_{x_{2}}\cdots\bullet_{x_{m}} \quad (\text{by Eq.~(\ref{eq:dota})})\\
=&\bullet_{x_{1}} \cdot \left( \sum_{i=2}^{m}\bullet_{x_{2}}\cdots\bullet_{x_{i}}\otimes\bullet_{x_{i}}\cdots\bullet_{x_{m}}+\lambda
\sum_{i=2}^{m-1}\bullet_{x_{2}}\cdots\bullet_{x_{i}}\otimes\bullet_{x_{i+1}}\cdots\bullet_{x_{m}}\right)
+\bullet_{x_{1}} \otimes\bullet_{x_{1}}\bullet_{x_{2}} \cdots\bullet_{x_{m}}\\
&+\lambda\bullet_{x_{1}}\ot \bullet_{x_{2}}\cdots\bullet_{x_{m}}
 \quad \quad \quad \quad \quad (\text{by the induction hypothesis})\\
=&  \sum_{i=2}^{m}\bullet_{x_{1}}\cdots\bullet_{x_{i}}\otimes\bullet_{x_{i}}\cdots\bullet_{x_{m}}
+\bullet_{x_{1}} \otimes\bullet_{x_{1}} \cdots\bullet_{x_{m}}
+\lambda\sum_{i=2}^{m-1}\bullet_{x_{1}}\cdots\bullet_{x_{i}}\otimes\bullet_{x_{i+1}}\cdots\bullet_{x_{m}}
+\lambda\bullet_{x_{1}}\ot \bullet_{x_{2}}\cdots\bullet_{x_{m}}\\
=&\sum_{i=1}^{m}\bullet_{x_{1}}\cdots\bullet_{x_{i}}\otimes\bullet_{x_{i}}\cdots\bullet_{x_{m}}
+\lambda\sum_{i=1}^{m-1}\bullet_{x_{1}}\cdots\bullet_{x_{i}}\otimes\bullet_{x_{i+1}}\cdots\bullet_{x_{m}},
\end{align*}
}
as required.
\end{proof}

\begin{lemma}\label{lem:colff}
Let $F_1, F_2\in \hrts$. Then
$$\col(F_1 F_2) = F_1 \cdot \col(F_2) + \col(F_1) \cdot F_2+\lambda (F_1\ot F_2).$$
\end{lemma}

\begin{proof}
It suffices to consider basis elements $F_1, F_2\in \rfs$ by linearity. We have two cases to consider.

\noindent{\bf Case 1.}
$\bre(F_1)=0$ or $\bre(F_2)=0$. In this case, without loss of generality, letting $\bre(F_1)=0$, then $F_1=\etree$ and by Eq.~(\mref{eq:dele}),
\begin{align*}
\col(F_1 F_2)&=\col(\etree F_2)=\col(F_2)-\lambda(\etree\ot F_2)+\lambda(\etree\ot F_2)\\
&=\col(F_2)-\lambda(\etree\ot \etree)\cdot F_2+\lambda(\etree\ot F_2)\\
&=\col(F_2)+\col(\etree)\cdot F_2+\lambda(\etree \ot F_2)\\
&=\etree\cdot\col(F_2)+\col(\etree)\cdot F_2+\lambda(\etree \ot F_2)\\
&=F_1 \cdot \col(F_2) + \col(F_1) \cdot F_2+\lambda (F_1\ot F_2).
\end{align*}

\noindent{\bf Case 2.}
$\bre(F_1)\geq1$ and $\bre(F_2)\geq1$. In this case, we prove the result by induction on the sum of breadths $\bre(F_1)+ \bre(F_2)\geq 2$. For the initial step of $\bre(F_1)+ \bre(F_2)= 2$, we have $F_1=T_1$ and $F_2=T_2$ for some decorated planar rooted trees $T_1, T_2\in \rts$. By Eq.~(\mref{eq:dele1}),
\begin{align*}
\col(F_1 F_2) =\col(T_1 T_2)&= T_1 \cdot \col(T_2) + \col(T_1) \cdot T_2+\lambda (T_1\ot T_2)\\
&=F_1 \cdot \col(F_2) + \col(F_1) \cdot F_2+\lambda (F_1\ot F_2).
\end{align*}
For the induction step of $\bre(F_1)+ \bre(F_2)\geq 3$, without loss of generality, we may suppose $\bre(F_2)\geq \bre(F_1)\geq 1$. If $\bre(F_1)=1$ and $\bre(F_2)\geq2$ , we may write $F_1=T_1$ for some decorated planar rooted trees $T_1\in \rts$. By Eq.~(\mref{eq:dele1}),
\begin{align*}
\col(F_1 F_2) =\col(T_1 F_2)&= T_1 \cdot \col(F_2) + \col(T_1) \cdot F_2+\lambda (T_1\ot F_2)\\
&=F_1 \cdot \col(F_2) + \col(F_1) \cdot F_2+\lambda (F_1\ot F_2).
\end{align*}
If $\bre(F_1)\geq2$, we can
write $F_1=T_{1}{F_{1}}'$ with $\bre(T_{1})=1$ and $\bre({F_{1}}') = \bre(F_1) -1 $. Then
\begin{align*}
&\ \col(F_1 F_2)=\col(T_{1}{F_{1}}' F_2)\\
&\ =T_{1}\cdot\col({F_{1}}'F_2)+\col(T_{1})\cdot({F_{1}}'F_2)+\lambda(T_1\ot {F_{1}}'F_2)\quad (\text{by Eq.~(\mref{eq:dele1})})\\
&\ =T_{1}\cdot\Big({F_{1}}' \cdot \col(F_2) + \col({F_{1}}') \cdot F_2+\lambda({F_{1}}'\ot F_2)\Big)+\col(T_{1})\cdot({F_{1}}'F_2)+\lambda(T_1\ot {F_{1}}'F_2) \\
&\hspace{8cm} (\text{by\ the\ induction\ hypothesis})\\
&\ =(T_{1}{F_{1}}')\cdot\col(F_2)+T_{1}\cdot\col({F_{1}}')\cdot F_2+\lambda(T_1{F_{1}}'\ot F_2)+\col(T_{1})\cdot ({F_{1}}' F_2)+\lambda(T_1\ot {F_{1}}'F_2) \\
&\ =(T_{1}{F_{1}}')\cdot \col(F_2)+\Big(T_{1}\cdot \col({F_{1}}') +\col(T_{1}) \cdot {F_{1}}'+\lambda (T_1\ot {F_{1}}' )\Big)\cdot F_2+\lambda(T_1{F_{1}}'\ot F_2)\\
&\ =(T_{1}{F_{1}}')\cdot \col(F_2)+\col(T_1{F_{1}}')\cdot F_2+\lambda(T_1{F_{1}}'\ot F_2)\\
&\ =F_1\cdot \col(F_2)+\col(F_1)\cdot F_2 +\lambda(F_1\ot F_2)\quad \quad \ \ \text{(by\ the\ induction\ hypothesis)}
\end{align*}
This completes the proof.
\end{proof}

The following lemma shows that $\hrts$ is closed under the coproduct $\col$.
\begin{lemma} \label{lem:cclosed}
For $F\in \hrts$,
\begin{equation}
\col(F)\in \hrts\ot\hrts.
\mlabel{fact:closed}
\end{equation}
\end{lemma}

\begin{proof}
We prove the result by induction on $\dep(F)\geq 0$ for basis elements $F\in \rfs$.
For the initial step of $\dep(F)=0$,
we have $F=\bullet_{x_1}\ldots\bullet_{x_{m}}$ for some $m\geq 0$, with the convention that $F=\etree$ when $m=0$.
If $m=0$, then
\begin{align*}
\col(F)=\col(\etree)=-\lambda(\etree\ot \etree)\in  \hrts \ot \hrts.
\end{align*}
If $m\geq 1$, then by Lemma~\mref{lem:rt11},
\begin{align*}
\col(F)=\col(\bullet_{x_1}\ldots\bullet_{x_{m}})
=&\sum_{i=1}^{m}\bullet_{x_{1}}\cdots\bullet_{x_{i}}\otimes\bullet_{x_{i}}\cdots\bullet_{x_{m}}+\lambda
\sum_{i=1}^{m-1}\bullet_{x_{1}}\cdots\bullet_{x_{i}}\otimes\bullet_{x_{i+1}}\cdots\bullet_{x_{m}}\\
&\in\hrts \ot \hrts.
\end{align*}

Suppose that Eq.~(\mref{fact:closed}) holds for $\dep(F) \leq n$ for an $n\geq 0$ and consider the case of $\dep(F) =n+1$.
For this case, we apply the induction on breadth $\bre(F)$. Since $\dep(F)=n+1 \geq 1$, we get $F\neq \etree$ and $\bre(F)\geq 1$.
If $\bre(F)=1$, since $\dep(F)\geq 1$, we have $F = B_{\omega}^{+}(\lbar{F})$ for some $\omega\in \Omega$ and $\lbar{F} \in \hrts$. By Eq.~(\mref{eq:dbp}), we have
\begin{align*}
\col(F)= \col\Big(B_{\omega}^{+}(\lbar{F})\Big) = F\ot (-\lambda\etree)+(\id\ot B_{\omega}^{+})\col(\lbar{F}).
\end{align*}
By the induction hypothesis on $\dep(F)$,
$$\col(\lbar{F})\in \hrts\ot \hrts \,\text{ and so }\, (\id\ot B_{\omega}^{+})\col(\lbar{F})\in \hrts\ot\hrts.$$
Moreover,  $-\lambda(F\ot \etree)\in \hrts \ot \hrts$ follows from $F\in \hrts$.
Hence
\begin{align*}
{F}\ot (-\lambda\etree)+(\id\ot B_{\omega}^{+})\col(\lbar{F})\in \hrts\ot\hrts.
\end{align*}

Assume that Eq.~(\mref{fact:closed}) holds for  $\dep(F) = n+1$ and $\bre(F)\leq m$, in addition to $\dep(F)\leq n$ by the first induction hypothesis, and consider the case of
$\dep(F) = n+1$ and $\bre(F) =m+1\geq 2$.
Then we may write $F=T_{1}T_{2}\cdots T_{m+1}$ for some $T_1,\ldots , T_{m+1} \in \rts$ and so
\begin{align*}
\col(F)=&\col(T_{1}T_{2}\cdots T_{m+1})\\
=&\ T_{1}\cdot \col(T_{2}\cdots T_{m+1})+\col(T_{1})\cdot (T_{2}\cdots T_{m+1})+\lambda(T_1\ot T_2\cdots T_{m+1}) \quad (\text{by Eq.~(\mref{eq:dele1})}).
\end{align*}
By the induction hypothesis on breadth, we have
$$ \col(T_{2}\cdots T_{m+1})\in  \hrts\ot \hrts\, \text{ and }\, \col(T_i) \in \hrts\ot \hrts,$$
whence by Eq.~(\mref{eq:dota}),
$$T_{1}\cdot \col(T_{2}\cdots T_{m+1})\in \hrts\ot \hrts \, \text{ and }\, \col(T_{1})\cdot (T_{2}\cdots T_{m+1})\in \hrts\ot \hrts.$$
Thus
$$\col(F)=\col(T_{1}T_{2}\cdots T_{m+1})\in \hrts\ot \hrts.$$
This completes the induction on the breadth and hence the induction on the depth.
\end{proof}

We now state our first main result in this subsection.

\begin{theorem}
The pair $(\hrts, \col)$ is a coalgebra (without counit).
\mlabel{thm:rt1}
\end{theorem}

\begin{proof}
By Lemma~\mref{lem:cclosed}, we only need to verify the the  coassociative law
\begin{equation}
(\id\otimes \col)\col(F)=(\col\otimes \id)\col(F)\,\text{ for } F\in \rfs,
\mlabel{eq:coass}
\end{equation}
which will be proved by induction on $\dep(F)\geq 0$.
For the initial step of $\dep(F)=0$, we have $F=\bullet_{x_{1}}\bullet_{x_{2}}\cdots\bullet_{x_{m}}$ for some $m\geq 0$, with the convention that $F=\etree$ if $m=0$.
When $m=0$, we have
\begin{align*}
(\id\otimes \col)\col(F)=(\id\otimes \col)\col(\etree)=-\lambda\etree\otimes \col(\etree)=\lambda^2 (\etree\ot \etree\ot \etree)=-\lambda \col(\etree)\ot \etree=(\col\otimes \id)\col(\etree).
\end{align*}
When $m\geq 1$, on the one hand,
\allowdisplaybreaks{
\begin{align*}
&\ (\id\ot\col)\col(\bullet_{x_{1}}\bullet_{x_{2}}\cdots\bullet_{x_{m}})\\
=&\ (\id\ot\col)\left(\sum_{i=1}^{m}\bullet_{x_{1}}\cdots\bullet_{x_{i}}
\otimes\bullet_{x_{i}}\cdots\bullet_{x_{m}}+\lambda\sum_{i=1}^{m-1}\bullet_{x_{1}}\cdots\bullet_{x_{i}}
\otimes\bullet_{x_{i+1}}\cdots\bullet_{x_{m}}\right)\quad (\text{by Lemma~\ref{lem:rt11}})\\
=&\sum_{i=1}^{m}\bullet_{x_{1}}\cdots\bullet_{x_{i}}
\otimes\col(\bullet_{x_{i}}\cdots\bullet_{x_{m}})+\lambda\sum_{i=1}^{m-1}\bullet_{x_{1}}\cdots\bullet_{x_{i}}
\otimes\col(\bullet_{x_{i+1}}\cdots\bullet_{x_{m}})\\
=&\sum_{i=1}^{m}\bullet_{x_{1}}\cdots\bullet_{x_{i}}
\otimes\left(\sum_{j=i}^{m} \bullet_{x_{i}}\cdots \bullet_{x_{j}}\ot \bullet_{x_{j}}\cdots \bullet_{x_{m}} +\lambda\sum_{j=i}^{m-1} \bullet_{x_{i}}\cdots \bullet_{x_{j}}\ot \bullet_{x_{j+1}}\cdots \bullet_{x_{m}}\right)\\
&+\lambda\sum_{i=1}^{m-1}\bullet_{x_{1}}\cdots\bullet_{x_{i}}
\otimes\left(\sum_{j=i+1}^{m} \bullet_{x_{i+1}}\cdots \bullet_{x_{j}}\ot \bullet_{x_{j}}\cdots \bullet_{x_{m}} +\lambda\sum_{j=i+1}^{m-1} \bullet_{x_{i+1}}\cdots \bullet_{x_{j}}\ot \bullet_{x_{j+1}}\cdots \bullet_{x_{m}}\right)\\
=&\sum_{i=1}^{m}\sum_{j=i}^{m}\bullet_{x_{1}}\cdots\bullet_{x_{i}}
\otimes \bullet_{x_{i}}\cdots \bullet_{x_{j}}\ot \bullet_{x_{j}}\cdots \bullet_{x_{m}} +\lambda\sum_{i=1}^{m}\sum_{j=i}^{m-1} \bullet_{x_{1}}\cdots\bullet_{x_{i}}
\otimes\bullet_{x_{i}}\cdots \bullet_{x_{j}}\ot \bullet_{x_{j+1}}\cdots \bullet_{x_{m}}\\
&+\lambda\sum_{i=1}^{m-1}\sum_{j=i+1}^{m}\bullet_{x_{1}}\cdots\bullet_{x_{i}}
\otimes \bullet_{x_{i+1}}\cdots \bullet_{x_{j}}\ot \bullet_{x_{j}}\cdots \bullet_{x_{m}} +\lambda^2\sum_{i=1}^{m-1}\sum_{j=i+1}^{m-1} \bullet_{x_{1}}\cdots\bullet_{x_{i}}
\otimes\bullet_{x_{i+1}}\cdots \bullet_{x_{j}}\ot \bullet_{x_{j+1}}\cdots \bullet_{x_{m}}\\
=&\sum_{i=1}^{m}\sum_{j=i}^{m}\bullet_{x_{1}}\cdots\bullet_{x_{i}}
\otimes \bullet_{x_{i}}\cdots \bullet_{x_{j}}\ot \bullet_{x_{j}}\cdots \bullet_{x_{m}} +\lambda\sum_{i=1}^{m}\sum_{j=i}^{m-1} \bullet_{x_{1}}\cdots\bullet_{x_{i}}
\otimes\bullet_{x_{i}}\cdots \bullet_{x_{j}}\ot \bullet_{x_{j+1}}\cdots \bullet_{x_{m}}\\
&+\lambda\sum_{i=1}^{m-1}\sum_{j=i+1}^{m}\bullet_{x_{1}}\cdots\bullet_{x_{i}}
\otimes \bullet_{x_{i+1}}\cdots \bullet_{x_{j}}\ot \bullet_{x_{j}}\cdots \bullet_{x_{m}} +\lambda^2\sum_{i=1}^{m-2}\sum_{j=i+1}^{m-1} \bullet_{x_{1}}\cdots\bullet_{x_{i}}
\otimes\bullet_{x_{i+1}}\cdots \bullet_{x_{j}}\ot \bullet_{x_{j+1}}\cdots \bullet_{x_{m}}
\end{align*}
}
On the other hand,
\allowdisplaybreaks{
\begin{align*}
&\ (\col \ot \id)\col(\bullet_{x_{1}}\bullet_{x_{2}}\cdots\bullet_{x_{m}})\\
=&\ (\col \ot \id)\left(\sum_{i=1}^{m}\bullet_{x_{1}}\cdots\bullet_{x_{i}}
\otimes\bullet_{x_{i}}\cdots\bullet_{x_{m}}+\lambda\sum_{i=1}^{m-1}\bullet_{x_{1}}\cdots\bullet_{x_{i}}
\otimes\bullet_{x_{i+1}}\cdots\bullet_{x_{m}}\right)\quad (\text{by Lemma~\ref{lem:rt11}})\\
=&\sum_{i=1}^{m}\col(\bullet_{x_{1}}\cdots\bullet_{x_{i}})
\otimes\bullet_{x_{i}}\cdots\bullet_{x_{m}}+\lambda\sum_{i=1}^{m-1}\col(\bullet_{x_{1}}\cdots\bullet_{x_{i}})
\otimes\bullet_{x_{i+1}}\cdots\bullet_{x_{m}}\\
=&\sum_{i=1}^{m}\left(\sum_{j=1}^{i} \bullet_{x_{1}}\cdots \bullet_{x_{j}}\ot \bullet_{x_{j}}\cdots \bullet_{x_{i}} +\lambda\sum_{j=1}^{i-1} \bullet_{x_{1}}\cdots \bullet_{x_{j}}\ot \bullet_{x_{j+1}}\cdots \bullet_{x_{i}}\right)\ot \bullet_{x_{i}}\cdots\bullet_{x_{m}}\\
&+\lambda\sum_{i=1}^{m-1}\left(\sum_{j=1}^{i}\bullet_{x_{1}}\cdots \bullet_{x_{j}}\ot \bullet_{x_{j}}\cdots \bullet_{x_{i}} +\lambda\sum_{j=1}^{i-1} \bullet_{x_{1}}\cdots \bullet_{x_{j}}\ot \bullet_{x_{j+1}}\cdots \bullet_{x_{i}}\right)\ot \bullet_{x_{i+1}}\cdots\bullet_{x_{m}}\\
=&\sum_{i=1}^{m}\sum_{j=1}^{i}\bullet_{x_{1}}\cdots\bullet_{x_{j}}
\ot  \bullet_{x_{j}}\cdots \bullet_{x_{i}}\ot \bullet_{x_{i}}\cdots \bullet_{x_{m}}
+\lambda\sum_{i=1}^{m}\sum_{j=1}^{i-1}\bullet_{x_{1}}\cdots\bullet_{x_{j}}
\ot  \bullet_{x_{j+1}}\cdots \bullet_{x_{i}}\ot \bullet_{x_{i}}\cdots \bullet_{x_{m}}\\
&+\lambda\sum_{i=1}^{m-1}\sum_{j=1}^{i}\bullet_{x_{1}}\cdots\bullet_{x_{j}}
\ot \bullet_{x_{j}}\cdots \bullet_{x_{i}}\ot \bullet_{x_{i+1}}\cdots \bullet_{x_{m}} +\lambda^2\sum_{i=1}^{m-1}\sum_{j=1}^{i-1}\bullet_{x_{1}}\cdots\bullet_{x_{j}}
\ot  \bullet_{x_{j+1}}\cdots \bullet_{x_{i}}\ot \bullet_{x_{i+1}}\cdots \bullet_{x_{m}}\\
=&\sum_{i=1}^{m}\sum_{j=1}^{i}\bullet_{x_{1}}\cdots\bullet_{x_{j}}
\ot  \bullet_{x_{j}}\cdots \bullet_{x_{i}}\ot \bullet_{x_{i}}\cdots \bullet_{x_{m}}
+\lambda\sum_{i=2}^{m}\sum_{j=1}^{i-1}\bullet_{x_{1}}\cdots\bullet_{x_{j}}
\ot  \bullet_{x_{j+1}}\cdots \bullet_{x_{i}}\ot \bullet_{x_{i}}\cdots \bullet_{x_{m}}\\
&+\lambda\sum_{i=1}^{m-1}\sum_{j=1}^{i}\bullet_{x_{1}}\cdots\bullet_{x_{j}}
\ot \bullet_{x_{j}}\cdots \bullet_{x_{i}}\ot \bullet_{x_{i+1}}\cdots \bullet_{x_{m}} +\lambda^2\sum_{i=2}^{m-1}\sum_{j=1}^{i-1}\bullet_{x_{1}}\cdots\bullet_{x_{j}}
\ot  \bullet_{x_{j+1}}\cdots \bullet_{x_{i}}\ot \bullet_{x_{i+1}}\cdots \bullet_{x_{m}}\\
=&\sum_{j=1}^{m}\sum_{i=j}^{m}\bullet_{x_{1}}\cdots\bullet_{x_{j}}
\ot  \bullet_{x_{j}}\cdots \bullet_{x_{i}}\ot \bullet_{x_{i}}\cdots \bullet_{x_{m}}
+\lambda\sum_{j=1}^{m-1}\sum_{i=j+1}^{m}\bullet_{x_{1}}\cdots\bullet_{x_{j}}
\ot  \bullet_{x_{j+1}}\cdots \bullet_{x_{i}}\ot \bullet_{x_{i}}\cdots \bullet_{x_{m}}\\
&+\lambda\sum_{j=1}^{m-1}\sum_{i=j}^{m-1}\bullet_{x_{1}}\cdots\bullet_{x_{j}}
\ot \bullet_{x_{j}}\cdots \bullet_{x_{i}}\ot \bullet_{x_{i+1}}\cdots \bullet_{x_{m}} +\lambda^2\sum_{j=1}^{m-2}\sum_{i=j+1}^{m-1}\bullet_{x_{1}}\cdots\bullet_{x_{j}}
\ot  \bullet_{x_{j+1}}\cdots \bullet_{x_{i}}\ot \bullet_{x_{i+1}}\cdots \bullet_{x_{m}}\\
=&\sum_{i=1}^{m}\sum_{j=i}^{m}\bullet_{x_{1}}\cdots\bullet_{x_{i}}
\otimes \bullet_{x_{i}}\cdots \bullet_{x_{j}}\ot \bullet_{x_{j}}\cdots \bullet_{x_{m}} +\lambda\sum_{i=1}^{m-1}\sum_{j=i+1}^{m}\bullet_{x_{1}}\cdots\bullet_{x_{i}}
\otimes \bullet_{x_{i+1}}\cdots \bullet_{x_{j}}\ot \bullet_{x_{j}}\cdots \bullet_{x_{m}} \\
&+\lambda\sum_{i=1}^{m}\sum_{j=i}^{m-1} \bullet_{x_{1}}\cdots\bullet_{x_{i}}
\otimes\bullet_{x_{i}}\cdots \bullet_{x_{j}}\ot \bullet_{x_{j+1}}\cdots \bullet_{x_{m}} +\lambda^2\sum_{i=1}^{m-2}\sum_{j=i+1}^{m-1} \bullet_{x_{1}}\cdots\bullet_{x_{i}}
\otimes\bullet_{x_{i+1}}\cdots \bullet_{x_{j}}\ot \bullet_{x_{j+1}}\cdots \bullet_{x_{m}}\\
&\hspace{7cm} (\text{by exchanging the index of $i$ and $j$}).
\end{align*}
}
Thus
\begin{align*}
(\id\ot\col)\col(\bullet_{x_{1}}\bullet_{x_{2}}\cdots\bullet_{x_{m}})=(\col \ot \id)\col(\bullet_{x_{1}}\bullet_{x_{2}}\cdots\bullet_{x_{m}}).
\end{align*}

Suppose that Eq.~(\mref{eq:coass}) holds for $\dep(F)\leq n$ for an $n\geq 0$ and consider the case of $\dep(F)=n+1$.
We now apply the induction on breadth. Since $\dep(F)=n+1\geq 1$, we have $F\neq \etree$ and $\bre(F)\geq 1$.
If $\bre(F)=1$, then we may write $F=B_{\omega}^{+}(\lbar{F})$ for some $\lbar{F}\in \rfs$ and $\omega\in \Omega$. Hence
\allowdisplaybreaks{
\begin{align*}
(\id\otimes\col)\col(F)=&\ (\id\otimes \col)\col(B_{\omega}^{+}(\lbar{F}))\\
=&\ (\id\otimes \col)\Big( F\otimes (-\lambda\etree)+(\id\otimes B_{\omega}^{+})\col(\lbar{F})\Big)\quad(\text{by Eq.~(\mref{eq:dbp})})\\
=&\ -\lambda F\otimes \col(\etree)+(\id\otimes (\col B_{\omega}^{+}))\col(\lbar{F})\\
=&\ \lambda^2F\ot \etree\ot \etree +(\id\otimes (\col B_{\omega}^{+}))\col(\lbar{F})  \quad(\text{by  Eq.~(\ref{eq:dele})})\\
=&\  \lambda^2F\ot \etree\ot \etree+ \Big(\id\otimes \big( B_{\omega}^{+}\ot (-\lambda\etree) + (\id \ot B_{\omega}^+)\col \big) \Big)\col(\lbar{F}) \quad(\text{by Eq.~(\mref{eq:cdbp})})\\
=&\  \lambda^2F\ot \etree\ot \etree+\Big(\id\otimes B_{\omega}^{+}\otimes (-\lambda\etree)+(\id\otimes \id\otimes B_{\omega}^{+})(\id\otimes \col)\Big)\col(\lbar{F})\\
=&\  \lambda^2F\ot \etree\ot \etree+\Big((\id\otimes B_{\omega}^{+})\col(\lbar{F})\Big)\otimes (-\lambda\etree)+(\id\otimes \id\otimes B_{\omega}^{+})(\id\otimes \col)\col(\lbar{F})\\
=&\  \lambda^2F\ot \etree\ot \etree +\Big((\id\otimes B_{\omega}^{+})\col(\lbar{F})\Big)\otimes (-\lambda\etree)+(\id\otimes \id\otimes B_{\omega}^{+})(\col\otimes \id)\col(\lbar{F})\\
&\hspace{4cm} (\text{by\ the\ induction\ hypothesis})\\
=&\  \lambda^2F\ot \etree\ot \etree+\Big((\id\otimes B_{\omega}^{+})\col(\lbar{F})\Big)\otimes (-\lambda\etree)+(\col\otimes B_{\omega}^{+})\col(\lbar{F})\\
=& \Big(F\ot (-\lambda\etree)+(\id\otimes B_{\omega}^{+})\col(\lbar{F})\Big)\otimes (-\lambda\etree)+(\col\otimes B_{\omega}^{+})\col(\lbar{F})\\
=&\ \col (F)\ot (-\lambda\etree)+(\col\otimes B_{\omega}^{+})\col(\lbar{F}) \quad(\text{by Eq.~(\mref{eq:dbp})}) \\
=&\ (\col\ot\id)\Big(F\ot (-\lambda\etree)+(\id\ot B_{\omega}^{+})\col(\lbar{F})\Big) \\
=&\ (\col\otimes \id)\col(F) \quad(\text{by Eq.~(\mref{eq:dbp})}).
\end{align*}
}
Assume that Eq.~(\mref{eq:coass}) holds for $\dep(F)=n+1$ and $\bre(F)\leq m$, in addition to $\dep(F)\leq n$ by the first induction hypothesis.
Consider the case when $\dep(F)=n+1$ and $\bre(F)=m+1 \geq 2$.
Then $F=F_{1}F_{2}$ for some $F_1, F_2\in \rfs$ with $0< \bre(F_1), \bre(F_2)< \bre(F)$.
Using the Sweedler notation, we may write
\begin{align*}
\Delta_{\epsilon}(F_1)=\sum_{(F_1)}F_{1(1)}\otimes F_{1(2)} \text{\ and \ } \Delta_{\epsilon}(F_2)=\sum_{(F_2)}F_{2(1)}\otimes F_{2(2)}.
\end{align*}
Then
\begin{align*}
&\ (\id\ot\col)\col(F_1F_2)\\
=&\ (\id\ot\col)(F_1 \cdot \col(F_2)+\col(F_1) \cdot F_2+\lambda(F_1\ot F_2))\quad(\text{by Lemma~\mref{lem:colff}})\\
=&\ (\id\ot\col)\left(\sum_{(F_2)}F_1F_{2(1)}\ot F_{2(2)}+\sum_{(F_1)}F_{1(1)}\ot F_{1(2)}F_2+\lambda(F_1\ot F_2) \right) \ \ (\text{by Eq.~(\ref{eq:dota})})\\
=&\ \sum_{(F_2)}F_1F_{2(1)}\ot\col(F_{2(2)})+\sum_{(F_1)}F_{1(1)}\ot\col(F_{1(2)}F_2)+\lambda(F_1\ot \col (F_2))\\
=&\ \sum_{(F_2)}F_1F_{2(1)}\ot\col(F_{2(2)})+\sum_{(F_1)}F_{1(1)}\ot \bigg(F_{1(2)}\cdot \col(F_2)+\col(F_{1(2)})\cdot F_2+\lambda (F_{1(2)}\ot F_2) \bigg)\\
&+\lambda(F_1\ot \col (F_2))\quad(\text{by Lemma~\mref{lem:colff}})\\
=&\ \sum_{(F_2)}F_1F_{2(1)}\ot \left( \sum_{(F_{2(2)})}F_{2(2)(1)}\ot F_{2(2)(2)}\right)+\sum_{(F_1)}F_{1(1)}\ot \left(\sum_{(F_2)}F_{1(2)}F_{2(1)}\ot F_{2(2)}\right)\\
&+\sum_{(F_1)}F_{1(1)}\ot \left(\sum_{(F_{1(2)})}F_{1(2)(1)}\ot F_{1(2)(2)}F_2\right) +\lambda\sum_{(F_1)}F_{1(1)}\ot F_{1(2)}\ot F_2+\lambda\sum_{(F_2)} F_1\ot F_{2(1)}\ot F_{2(2)}\\
&\hspace{5cm}(\text{by the Sweedler notation and Eq.~(\ref{eq:dota})})\\
=&\ \sum_{(F_2)}\sum_{(F_{2(2)})} F_1F_{2(1)}\ot F_{2(2)(1)}\ot F_{2(2)(2)}+\sum_{(F_1)}\sum_{(F_2)}F_{1(1)} \ot F_{1(2)}F_{2(1)}\ot F_{2(2)}\\
&+\sum_{(F_1)}\sum_{(F_{1(2)})}F_{1(1)}\ot F_{1(2)(1)} \ot F_{1(2)(2)}F_2 +\lambda \sum_{(F_1)}F_{1(1)}\ot F_{1(2)}\ot F_2+\lambda \sum_{(F_2)} F_1\ot F_{2(1)}\ot F_{2(2)}.
\end{align*}
Similarly, we have
\begin{align*}
&\ (\col\ot \id)\col(F_1F_2)\\
=&\  \sum_{(F_2)}\sum_{(F_{2(1)})}F_1F_{2(1)(1)}\ot F_{2(1)(2)} \ot F_{2(2)}+\sum_{(F_1)}\sum_{(F_2)}F_{1(1)} \ot F_{1(2)}F_{2(1)}\ot F_{2(2)}\\
&+\sum_{(F_1)}\sum_{(F_{1(1)})}F_{1(1)(1)}\ot F_{1(1)(2)} \ot F_{1(2)}F_2+\lambda \sum_{(F_1)}F_{1(1)}\ot F_{1(2)}\ot F_2+\lambda \sum_{(F_2)} F_1\ot F_{2(1)}\ot F_{2(2)}.
\end{align*}
By the induction hypothesis, we have
\begin{align*}
(\id\ot\col)\col(F_1)=&\sum_{(F_1)}\sum_{(F_{1(2)})}F_{1(1)}\ot F_{1(2)(1)} \ot F_{1(2)(2)}\\
=&(\col\ot \id)\col(F_1)=\sum_{(F_1)}\sum_{(F_{1(1)})}F_{1(1)(1)}\ot F_{1(1)(2)} \ot F_{1(2)}
\end{align*}
and
\begin{align*}
(\id\ot\col)\col(F_2)=&\sum_{(F_2)}\sum_{(F_{2(2)})}F_{2(1)}\ot F_{2(2)(1)} \ot F_{2(2)(2)}\\
=&(\col\ot \id)\col(F_2)=\sum_{(F_2)}\sum_{(F_{2(1)})}F_{2(1)(1)}\ot F_{2(1)(2)} \ot F_{2(2)}.
\end{align*}
Thus
\begin{align*}
(\id\ot\col)\col(F_1F_2)=(\col\ot \id)\col(F_1F_2).
\end{align*}
This completes the induction on the breadth and hence the induction on the depth.
\end{proof}

Now we arrive at our second main result in this subsection.
\begin{theorem}
The quadruple $(\hrts,\, \mul, \etree, \,\col)$ is an  $\epsilon$-unitary bialgebra of weight $\lambda$.
\mlabel{thm:rt2}
\end{theorem}

\begin{proof}
Note that the triple $(\hrts,\, \mul, \etree)$ is a unitary algebra.
Then the result follows from Lemma~\mref{lem:colff} and Theorem~\mref{thm:rt1}.
\end{proof}

\section{Free $\Omega$-cocycle infinitesimal unitary bialgebras}\label{sec:ope}
In this section, we conceptualize the combination of operated algebras and weighted infinitesimal unitary bialgebras,
and show that $\hrts$ is a free object in such category.
Let us start with the following concepts.

\begin{defn}\cite[Section~1.2]{Guo09} Let $\Omega$ be a nonempty set.
\begin{enumerate}
\item
 An {\bf $\Omega$-operated monoid } is a monoid $M$ together with a set of operators $P_{\omega}: M\to M$, $\omega\in \Omega$.

\item
An {\bf $\Omega$-operated  algebra } is an algebra $A$ together with a set of linear operators $P_{\omega}: A\to A$, $\omega\in \Omega$.
\end{enumerate}
\end{defn}

\begin{defn}\cite[Definition~3.17]{GZ}
Let $\lambda$ be a given element of $\bfk$.
\begin{enumerate}
\item  An  {\bf $\Omega$-operated  $\epsilon$-bialgebra} {\bf of weight $\lambda$} is an $\epsilon$-bialgebra $H$ of weight $\lambda$ together with a set of linear operators $P_{\omega}: H\to H$, $\omega\in \Omega$.
\mlabel{it:def1}
\item Let $(H,\, \{P_{\omega}\mid \omega \in \Omega\})$ and $(H',\,\{P'_{\omega}\mid \omega\in \Omega\})$ be two $\Omega$-operated $\epsilon$-bialgebras of weight $\lambda$. A linear map $\phi : H\rightarrow H'$ is called an {\bf $\Omega$-operated $\epsilon$-bialgebra morphism} if $\phi$ is a morphism of $\epsilon$-bialgebras of weight $\lambda$ and $\phi \circ P_\omega = P'_\omega \circ\phi$ for $\omega\in \Omega$.
\end{enumerate}
\end{defn}

By Remark~\mref{remk:4rem} (\mref{remk:units}), the unit of an infinitesimal unitary bialgebra is a group like element of weight $-\lambda$. Involving a weighted 1-cocycle condition, we then propose

\begin{defn}
\begin{enumerate}
\item \mlabel{it:def3}
An {\bf $\Omega$-cocycle $\epsilon$-unitary bialgebra} {\bf of weight $\lambda$} is an $\Omega$-operated  $\epsilon$-unitary bialgebra $(H,\,m,\, 1,\, \Delta, \{P_{\omega}\mid \omega \in \Omega\})$ of weight $\lambda$ satisfying the weighted 1-cocycle condition:
\begin{equation}
\Delta P_{\omega} =  P_{\omega} \otimes (-\lambda 1) + (\id\otimes P_{\omega}) \Delta \quad \text{ for all } \omega\in \Omega.
\mlabel{eq:eqiterated}
\end{equation}

\item The {\bf free $\Omega$-cocycle $\epsilon$-unitary bialgebra of weight $\lambda$ on a set $X$} is an $\Omega$-cocycle $\epsilon$-unitary bialgebra
$(H_{X},\,m_{X}, \,1_{X}, \Delta_{X}, \,\{P_{\omega}\mid \omega \in \Omega\})$ of weight $\lambda$ together with a set map $j_X: X \to H_{X}$  with the property that,
for any $\Omega$-cocycle $\epsilon$-unitary bialgebra $(H,\,m,\,1, \Delta,\,\{P'_{\omega}\mid \omega \in \Omega\})$ of weight $\lambda$ and any set map
$f: X\to H$ whose images are {\bf group like} (that is, $\Delta (f(x))=f(x)\ot  f(x)$ for $x \in X$), there is a unique morphism $\free{f}:H_X\to H$ of
$\Omega$-operated $\epsilon$-unitary bialgebras such that $\free{f}\circ j_X =f$.\mlabel{it:def4}
\end{enumerate}
\mlabel{defn:xcobi}
\end{defn}

\begin{remark}
Note the subtle difference between the weighted cocycle condition Eq.~(\mref{eq:eqiterated}) and the $\epsilon$-cocycle condition in~\cite[Definition~3.17]{GZ}:
\begin{equation*}
\Delta P_{\omega} =  \id \otimes 1 + (\id\otimes P_{\omega}) \Delta \quad \text{ for all } \omega\in \Omega.
\end{equation*}
\end{remark}

The following results generalizes the universal properties which were studied in~\cite{CK98, Fo3, Guo09, Moe01, ZGG16}. Recall from Eq.~(\mref{eq:rdeff}) that
\begin{align*}
\rfs = \lim_{\longrightarrow} \calf_n=\bigcup_{n=0}^{\infty} \calf_n.
\end{align*}

\begin{theorem}\label{thm:propm}
Let $j_{X}: X\hookrightarrow \rfs$, $x \mapsto \bullet_{x}$ be the natural embedding and $m_{\RT}$ be the concatenation product.
\begin{enumerate}
\item
The quadruple $(\rfs, \,\mul,\, \etree, \, \{B_{\omega}^+\mid \omega\in \Omega\})$ together with the $j_X$ is the free $\Omega$-operated monoid on $X$.
\mlabel{it:fomonoid}
\item
The quadruple $(\hrts, \,\mul,\,\etree, \, \{B_{\omega}^+\mid \omega\in \Omega\})$ together with the $j_X$ is the free $\Omega$-operated unitary algebra on $X$.
\mlabel{it:fualg}
\item
 The quintuple $(\hrts, \,\mul,\,\etree, \, \col, \, \{B_{\omega}^+\mid \omega\in \Omega\})$ together with the $j_X$ is the free $\Omega$-cocycle $\epsilon$-unitary bialgebra of weight $\lambda$ on $X$.
 \mlabel{it:fubialg}
\end{enumerate}

\end{theorem}

\begin{proof}

(\mref{it:fomonoid}) We only need to verify that $(\rfs, \, \{B_{\omega}^+\mid \omega\in \Omega\})$ satisfies the universal property. Let $(S,  \{P_\omega\mid \omega\in \Omega\})$ be a given $\Omega$-operated monoid  and $f:X\rightarrow S$ a given set map.
We will use induction on $n$ to construct a unique sequence of monoid homomorphisms
\begin{align*}
 \bar{f}_{n}:\calf_n\to S, n\geq 0.
\end{align*}

For the initial step of $n=0$, by the universal property of the free monoid $M(\bullet_{X})$,
the map $\bullet_X \to S$, $\bullet_x\mapsto f(x)$ extends to a unique monoid homomorphism $\bar{f}_0: M(\bullet_{X}) \to S$.
Assume that $\bar{f}_k: \calf_{k}\to S$ has been defined for a $k\geq 0$ and define the set map
\begin{align*}
\bar{f}_{k+1}: \bullet_X\sqcup (\sqcup_{\omega\in \Omega} B_{\omega}^{+}({\calf_{k}})) \to S, \, \bullet_x \mapsto f(x), \, B_{\omega}^{+}({\lbar{F}}) \mapsto  P_\omega(\bar{f}_{k}(\lbar{F})),
\end{align*}
where $x\in X$, $\omega\in \Omega$ and $\lbar{F}\in \calf_{k}$.
Again by the universal property of the free monoid
$M(\bullet_X\sqcup \sqcup_{\omega\in \Omega} B_{\omega}^{+}({\calf_{k}}))$, $\bar{f}_{k+1}$ is extended to a unique monoid homomorphism
\[
\bar{f}_{k+1}: \calf_{k+1} = M(\bullet_X\sqcup (\sqcup_{\omega\in \Omega} B_{\omega}^{+}({\calf_{k}}))) \to S.
\]
Define
$$\bar{f}:= \dirlim \bar{f}_n: \rfs  \to S.$$
Then by the above construction, $\bar{f}$ is the required homomorphism of $\Omega$-operated monoids and the unique one such that $\bar{f} \circ j_X = f$.

(\mref{it:fualg}) It directly follows from Item (\mref{it:fomonoid}).

(\mref{it:fubialg}) By Theorem~\mref{thm:rt2}, $(\hrts, \,\mul,\,\etree, \col)$ is an $\epsilon$-unitary bialgebra of weight $\lambda$.
Morover by Eq.~(\mref{eq:dbp}), $(\hrts, \,\mul,\,\etree, \col,\,\{B_{\omega}^+\mid \omega\in \Omega\})$ is an $\Omega$-cocycle $\epsilon$-unitary bialgebra of weight $\lambda$.

For the freeness, let $(H,\, m,\,1, \Delta,\, \{P_{\omega}\mid \omega \in \Omega\})$ be an $\Omega$-cocycle $\epsilon$-unitary bialgebra of weight $\lambda$ and $f: X\to H$ a set map such that
$$\Delta (f(x))= f(x)\ot f(x)\quad \text{ for all } x\in X.$$
In particular, $(H,\, m,\, 1,\, \{P_{\omega}\mid \omega \in \Omega\})$ is an $\Omega$-operated unitary algebra.
By Item (\mref{it:fualg}), there exists a unique $\Omega$-operated unitary algebra morphism $\free{f}:\hrts \to H$ such that $\free{f}\circ j_X={f}$. It remains to check the compatibility of the coproducts $\Delta$ and $\col$ for which we verify
\begin{equation}
\Delta \free{f} (F)=(\free{f}\ot \free{f}) \col (F)\quad \text{for all } F\in \rfs,
\mlabel{eq:copcomp}
\end{equation}
by induction on the depth $\dep(F)\geq 0$.
For the initial step of $\dep(F)=0$,
we have $F = \bullet_{x_1} \cdots \bullet_{x_m}$ for some $m\geq 0$, with the convention that $F=\etree$ when $m=0$.
If $m=0$, then by Remark~\mref{remk:4rem}~(\mref{remk:units}) and Eq.~(\mref{eq:dele}),
\begin{align*}
\Delta  \bar{f} (F) &= \Delta  \bar{f} (\etree)=\Delta (1)=-\lambda (1\ot 1) =-\lambda \bar{f} (\etree)\ot\bar{f} (\etree) =(\bar{f}\otimes\bar{f})(-\lambda\etree\ot \etree)=(\bar{f}\otimes\bar{f})\col(\etree).
\end{align*}
If $m\geq 1$, then
\allowdisplaybreaks{
\begin{align*}
&\ \Delta\free{f}(\bullet_{x_1} \cdots \bullet_{x_m})
= \Delta \Big( \free{f}(\bullet_{x_1})\cdots\free{f}(\bullet_{x_{m}}) \Big)\\
=&\ \sum_{i=1}^m \Big(\free{f}(\bullet_{x_{1}})\cdots\free{f}(\bullet_{x_{i-1}})\Big)\cdot
\Delta\big(\free{f}(\bullet_{x_{i}})\big) \cdot \Big(\free{f}(\bullet_{x_{i+1}})\cdots\free{f}(\bullet_{x_{m}})\Big)\\
&+\lambda\sum_{i=1}^{m-1}\free{f}(\bullet_{x_{1}})\cdots\free{f}(\bullet_{x_{i}}) \ot \free{f}(\bullet_{x_{i+1}})\cdots \free{f}(\bullet_{x_{m}})
\quad \text{(by Eq.~(\ref{eq:cocycle}))}\\
=&\ \sum_{i=1}^m \Big(\free{f}(\bullet_{x_{1}})\cdots\free{f}(\bullet_{x_{i-1}})\Big)\cdot
\Big(\free{f}(\bullet_{x_{i}})\ot \free{f}(\bullet_{x_{i}})\Big)\cdot \Big(\free{f}(\bullet_{x_{i+1}})\cdots\free{f}(\bullet_{x_{m}})\Big)\\
&+\lambda\sum_{i=1}^{m-1}\free{f}(\bullet_{x_{1}})\cdots\free{f}(\bullet_{x_{i}}) \ot \free{f}(\bullet_{x_{i+1}})\cdots \free{f}(\bullet_{x_{m}})\\
&\hspace{2cm}(\text{by}\,\, \Delta(\free{f}(\bullet_{x_{i}}))=\Delta(f(x_i))= f(x_i)\ot f(x_i) = \bar{f}(x_i)\ot \bar{f}(x_i))\\
=&\  \sum_{i=1}^m \free{f}(\bullet_{x_{1}})\cdots\cdot
\free{f}(\bullet_{x_{i}})\ot \free{f}(\bullet_{x_{i}})\cdots\free{f}(\bullet_{x_{m}})
+\lambda\sum_{i=1}^{m-1}\free{f}(\bullet_{x_{1}})\cdots\free{f}(\bullet_{x_{i}}) \ot \free{f}(\bullet_{x_{i+1}})\cdots \free{f}(\bullet_{x_{m}})\\
=&\  \sum_{i=1}^m \free{f}(\bullet_{x_{1}}\cdots\cdot
\bullet_{x_{i}})\ot \free{f}(\bullet_{x_{i}}\cdots\bullet_{x_{m}})
+\lambda\sum_{i=1}^{m-1}\free{f}(\bullet_{x_{1}}\cdots\bullet_{x_{i}}) \ot \free{f}(\bullet_{x_{i+1}}\cdots \bullet_{x_{m}})\\
&\hspace{4cm}(\text{by $\free{f}$ being a unitary algebra morphism})\\
=&\ (\free{f}\ot\free{f})
\left(\sum_{i=1}^{m}\bullet_{x_{1}}\cdots\bullet_{x_{i}}\otimes\bullet_{x_{i}}\cdots\bullet_{x_{m}}
+\lambda\sum_{i=1}^{m-1}\bullet_{x_{1}}\cdots\bullet_{x_{i}}\otimes\bullet_{x_{i+1}}\cdots\bullet_{x_{m}}\right)\\
=&\ (\free{f}\ot\free{f})\col(\bullet_{x_1} \cdots \bullet_{x_m})\quad(\text{by Lemma~\ref{lem:rt11}}).
\end{align*}
}

Suppose Eq.~(\mref{eq:copcomp}) holds for $\dep(F)\leq n$ for an $n\geq 0$ and consider the case of $\dep(F)=n+1$.
For this case we apply the induction on the breadth $\bre(F)$. Since $\dep(F)=n+1\geq1$, we have $F\neq \etree$ and $\bre(F)\geq 1$.
If $\bre(F)=1$, we have $F=B_{\omega}^+(\lbar{F})$ for some $\lbar{F}\in\rfs$ and $\omega\in \Omega$.  Then
\begin{align*}
\Delta \free{f}(F)&=\Delta \free{f} (B_{\omega}^{+}(\lbar{F}))=\Delta P_\omega(\free{f} (\lbar{F}))\quad(\text{by $\bar{f}$ being an operated unitary algebra morphism})\\
&=P_\omega(\free{f}(\lbar{F}))\ot (-\lambda1)+ (\id\ot P_{\omega})\Delta(\free{f} (\lbar{F}))
\quad(\text{by Eq.~(\mref{eq:eqiterated})})\\
&=P_\omega(\free{f}(\lbar{F}))\ot (-\lambda1)+ (\id\ot P_{\omega})(\free{f}\ot \free{f}) \col (\lbar{F}) \quad(\text{by the induction hypothesis on~}\dep(F)) \\
&=P_\omega(\free{f}(\lbar{F}))\ot (-\lambda1)+ (\free{f}\ot P_{\omega}\free{f}) \col (\lbar{F})\\
&=\free{f}(B_{\omega}^{+}(\lbar{F}))\ot (-\lambda1)+ (\free{f}\ot \free{f}B_{\omega}^+) \col (\lbar{F})
\quad(\text{by $\bar{f}$ being an operated unitary algebra morphism}) \\
&=(\free{f}\ot \free{f})\Big(B_{\omega}^{+}(\lbar{F})\ot (-\lambda\etree)+(\id\ot B_{\omega}^+)\col (\lbar{F})\Big) \\
&=(\free{f}\ot \free{f}) \col (B_{\omega}^+(\lbar{F}))\quad(\text{by Eq.~(\mref{eq:dbp})}) \\
&=(\free{f}\ot \free{f}) \col (F).
\end{align*}

Assume that Eq.~(\mref{eq:copcomp}) holds for $\dep(F)=n+1$ and $\bre(F)\leq m$, in addition to $\dep(F)\leq n$ by the first induction hypothesis, and consider the case when $\dep(F)=n+1$ and $\bre(F)=m+1\geq 2$. Then we can write $F=F_{1}F_{2}$ for some $F_{1},F_{2}\in\rfs$ with $0< \bre(F_{1}), \bre(F_{2}) < m+1$.
Using the Sweedler notation, we can write
\begin{align*}
\Delta_{\epsilon}(F_1)=\sum_{(F_1)}F_{1(1)}\otimes F_{1(2)} \text{\ and \ } \Delta_{\epsilon}(F_2)=\sum_{(F_2)}F_{2(1)}\otimes F_{2(2)}.
\end{align*}
By the induction hypothesis on the breadth, we have
\begin{align*}
&\Delta(\free{f}(F_{1}))=(\free{f} \ot \free{f})\col(F_{1})=\sum_{(F_{1})}\free{f}(F_{1(1)})\ot\free{f} (F_{1(2)}),\\
&\Delta(\free{f}(F_{2}))=(\free{f} \ot \free{f})\col(F_{2})=
\sum_{(F_{2})}\free{f}(F_{2(1)})\ot\free{f} (F_{2(2)}).
\end{align*}
Thus
\allowdisplaybreaks{
\begin{align*}
\Delta \free{f}(F)=&\ \Delta \free{f} (F_{1}F_{2})=\Delta(\free{f}(F_1)\free{f}(F_2))\\
=&\ \free{f}(F_{1})\cdot \Delta (\free{f}(F_{2}))+\Delta(\free{f}(F_{1})) \cdot \free{f}(F_{2})+\lambda\free{f}(F_{1})\ot \free{f}(F_{2})\quad(\text{by Eq.~(\mref{eq:cocycle})})\\
=&\ \free{f}(F_{1})\cdot \bigg(\sum_{(F_{2})}\free{f}(F_{2(1)})\ot \free{f}(F_{2(2)})\bigg)+
\bigg(\sum_{(F_{1})}\free{f}(F_{1(1)})\ot\free{f} (F_{1(2)})\bigg) \cdot \free{f}(F_{2})+\lambda\free{f}(F_{1})\ot \free{f}(F_{2})\\
=&\ \sum_{(F_{2})}\free{f}(F_{1})\free{f}(F_{2(1)})\ot\free{f}(F_{2(2)})
+\sum_{(F_{1})}\free{f}(F_{1(1)})\ot\free{f} (F_{1(2)})\free{f}(F_{2})+\lambda\free{f}(F_{1})\ot \free{f}(F_{2}) \, (\text{by Eq.~(\ref{eq:dota})})\\
=&\ \sum_{(F_{2})}\free{f}(F_{1}F_{2(1)})\ot\free{f}(F_{2(2)})+\sum_{(F_{1})}\free{f}(F_{1(1)})\ot\free{f} (F_{1(2)}F_{2})+\lambda\free{f}(F_{1})\ot \free{f}(F_{2})\\
=&\ (\free{f}\ot \free{f})\bigg(\sum_{(F_{2})}F_{1}F_{2(1)}\ot F_{2(2)}\bigg)+(\free{f}\ot \free{f})\bigg( \sum_{(F_{1})}F_{1(1)}\ot F_{1(2)}F_{2}\bigg)+(\free{f}\ot \free{f})(\lambda F_1\ot F_2)\\
=&\ (\free{f}\ot \free{f})\left(\sum_{(F_{2})}F_{1}F_{2(1)}\ot F_{2(2)}
+\sum_{(F_{1})}F_{1(1)}\ot F_{1(2)}F_{2} +\lambda F_1\ot F_2\right)\\
=&\ (\free{f}\ot \free{f})\left(F_{1}\cdot \sum_{(F_{2})}F_{2(1)}\ot F_{2(2)}+ \Big(\sum_{(F_{1})}F_{1(1)}\ot F_{1(2)}\Big) \cdot F_{2} +\lambda F_1\ot F_2\right) \quad (\text{by Eq.~(\ref{eq:dota})})\\
=&\ (\free{f} \ot \free{f})(F_{1} \cdot \col(F_{2})+\col(F_{1})\cdot F_{2}+\lambda(F_1\ot F_2))\\
=&\ (\free{f}\ot \free{f})\col(F_{1}F_{2})  \quad(\text{by Lemma~\mref{lem:colff}})\\
=&\ (\free{f}\ot \free{f})\col(F).
\end{align*}
}
This completes the induction on the breadth and hence the induction on the depth.
\end{proof}


Let $X=\emptyset$. Then we obtain a freeness of $\hck(\emptyset, \Omega)$, which is the infinitesimal version of decorated noncommutative Connes-Kreimer Hopf algebra by Remark~\mref{re:3ex}~(\mref{it:2ex}).
\begin{coro}
The quintuple $(\hck(\emptyset, \Omega), \,\mul,\,\etree,\, \col,\,\{B_{\omega}^+\mid \omega\in \Omega\})$ is the free $\Omega$-cocycle $\epsilon$-unitary bialgebra of weight $\lambda$ on the empty set, that is, the initial object in the category of $\Omega$-cocycle $\epsilon$-unitary bialgebras of weight $\lambda$.
\mlabel{cor:rt16}
\end{coro}

\begin{proof}
It follows from Theorem~\mref{thm:propm}~(\mref{it:fubialg}) by taking $X=\emptyset$.
\end{proof}

Taking $\Omega$ to be singleton in Corollary~\mref{cor:rt16},
all vertices of planar rooted forests have the same decoration.
In other words, in this case planar rooted forests have no decorations and that are precisely the one
in the classical noncommutative Connes-Kreimer Hopf algebra, introduced by Foissy~\mcite{Foi02} and Holtkamp~\mcite{Hol03}.

\begin{coro}
Let $\mathcal{F}$ be the set of planar rooted forests without decorations.
Then the quintuple $(\bfk\calf, \,\mul,\,\etree,\, \col,\, B^+)$ is the free cocycle $\epsilon$-unitary bialgebra of weight $\lambda$ on the empty set, that is, the initial object in the category of $\Omega$-cocycle $\epsilon$-unitary bialgebras of weight $\lambda$.
\mlabel{coro:rt16}
\end{coro}

\begin{proof}
It follows from Corollary~\mref{cor:rt16}  by taking $\Omega$ to be a singleton set.
\end{proof}

\smallskip

\noindent {\bf Acknowledgments}:
This work was supported by the National Natural Science Foundation
of China (Grant No.\@ 11771191and 11861051).
Yi Zhang would like to express his gratitude to Professor Foissy for telling the facts on the 1-cocycle condition, which make it possible to construct a free $\Omega$-cocycle $\epsilon$-unitary bialgebra on a set $X$.

\end{document}